\documentclass{article}
\usepackage{bm}
\usepackage{tikz,mathpazo}
\usepackage{pgf}
\usepackage[top=2.5cm, bottom=2.5cm, left=3cm, right=3cm]{geometry}   
\usepackage{indentfirst}
\usepackage{graphicx}
\usepackage{graphics}
\usepackage[toc,page,title,titletoc,header]{appendix}
\usepackage{bm}
\usepackage{bbm}
\usepackage{listings}  
\usepackage{amsmath}
\usepackage{setspace} 
\usepackage{indentfirst}
\usepackage{caption}
\usepackage{multirow} 
\usepackage{lipsum,multicol}
\usepackage{pdfpages}
\usepackage{float}
\usepackage{amsthm}
\usepackage{pdfpages}
\usepackage{url}
\usepackage{colortbl}
\usepackage{subfigure}
\usepackage{epsfig}
\usepackage{epstopdf}
\usetikzlibrary{shapes.geometric, arrows}
\usepackage{fancyhdr}
\usepackage{abstract}
\usepackage{tikz,mathpazo}
\usetikzlibrary{shapes.geometric, arrows}
\usepackage{amssymb}
\usepackage{latexsym}
\usepackage{verbatim}
\usepackage[numbers]{natbib} 
\usepackage{booktabs}

\usepackage{empheq}
\usepackage{xcolor}
\usepackage{mdframed}
\usepackage{framed}

\newmdenv[backgroundcolor=gray!20,innertopmargin=0pt,innerbottommargin=0pt,innerleftmargin=0pt,innerrightmargin=0pt,linewidth=0pt]{greybox}

\usepackage{hyperref}
\hypersetup{
	colorlinks=true,
	linkcolor=blue,
	filecolor=magenta,      
	urlcolor=cyan,
	citecolor = blue,
}

\allowdisplaybreaks[4]

\newcommand{\inner}[1]{\left\langle #1 \right\rangle}
\newcommand{\norm}[1]{\left\Vert #1\right\Vert}
\newcommand{\bb}[1]{\mathbb{#1}}

\newcommand{\conv}[0]{\mathrm{conv}}

\newcommand{\X}{{ \ca{X} }}
\newcommand{\Y}{{ \ca{Y} }}

\newcommand{\ca}[1]{\mathcal{#1}}

\newcommand{\M}[0]{\mathcal{M}}

\newcommand{\prox}{{\mathrm{prox}}}

\newcommand{\NX}{\ca{N}_{\X}}

\newcommand{\vk}{{v_{k} }}

\newcommand{\vkp}{{v_{k+1} }}

\newcommand{\xk}{{x_{k} }}
\newcommand{\yk}{{y_{k} }}
\newcommand{\xkp}{{x_{k+1} }}

\newcommand{\ykp}{{y_{k+1} }}
\newcommand{\zk}{{z_{k} }}

\newcommand{\wk}{{w_{k} }}
\newcommand{\wkp}{{w_{k+1} }}

\newcommand{\D}{\ca{D}}

\newcommand{\grad}{{\mathit{grad}\,}}

\newcommand{\Rn}{\mathbb{R}^n}
\newcommand{\Rnp}{\mathbb{R}^{n\times p}}
\newcommand{\Rp}{\mathbb{R}^p}

\newcommand{\Rs}{\mathbb{R}^s}
\newcommand{\Rq}{\mathbb{R}^q}

\newcommand{\Tregu}{\ca{T}_{regu}}

\newcommand{\y}{{y}}

\newcommand{\pprox}{\widetilde{\prox}}

\newtheorem{theo}{Theorem}[section]
\newtheorem{lem}[theo]{Lemma}
\newtheorem{prop}[theo]{Proposition}

\newtheorem{cond}[theo]{Condition}
\newtheorem{coro}[theo]{Corollary}

\newtheorem{defin}[theo]{Definition}
\newtheorem{rmk}[theo]{Remark}
\newtheorem{assumpt}[theo]{Assumption}

\usepackage[figuresright]{rotating}
\usepackage{pdflscape}

\DeclareMathOperator*{\argmin}{arg\,min}

\usepackage{caption}
\usepackage{algorithm}
\usepackage{algpseudocode}
\usepackage{longtable}
\usepackage{appendix}
\usepackage{tikz-cd}
\usepackage{array}

\numberwithin{equation}{section}

\title{Developing Lagrangian-based Methods for Nonsmooth Nonconvex Optimization}
\author{Nachuan Xiao\thanks{Institute of Operational Research and Analytics, National University of Singapore, Singapore. (xnc@lsec.cc.ac.cn).}, ~
	Kuangyu Ding\thanks{Department of Mathematics, National University of Singapore, Singapore. (kuangyud@u.nus.edu).},  ~
	Xiaoyin Hu\thanks{School of Computer and Computing Science, Hangzhou City University, Hangzhou, 310015, China.  (hxy@amss.ac.cn).},  ~
	Kim-Chuan Toh\thanks{Department of Mathematics, and Institute of Operations Research and Analytics, National University of Singapore, Singapore 119076. (mattohkc@nus.edu.sg).}}

\begin{document}
	\maketitle
	
	\begin{abstract}
		In this paper, we consider the minimization of a nonsmooth nonconvex objective function $f(x)$ over a closed convex subset $\mathcal{X}$ of $\mathbb{R}^n$, with additional nonsmooth nonconvex constraints $c(x) = 0$. We develop a unified framework for developing Lagrangian-based methods, which takes a single-step update to the primal variables by some subgradient methods in each iteration. These subgradient methods are ``embedded'' into our framework, in the sense that they are incorporated as black-box updates to the primal variables. We prove that our proposed framework inherits the global convergence guarantees from these embedded subgradient methods under mild conditions. In addition, we show that our framework can be extended to solve constrained optimization problems with expectation constraints.  Based on the proposed framework, we show that a wide range of existing stochastic subgradient methods, including the proximal SGD, proximal momentum SGD, and proximal ADAM, can be embedded into Lagrangian-based methods. Preliminary numerical experiments on deep learning tasks illustrate that our proposed framework yields efficient variants of Lagrangian-based methods with convergence guarantees for nonconvex nonsmooth constrained optimization problems.

	\end{abstract}

\section{Introduction}
In this paper, we consider the following nonsmooth constrained optimization problem (NCP),
	\begin{equation}
		\label{Prob_Ori}
		\tag{NCP}
		\begin{aligned}
			\min_{x \in \X}\quad & f(x),\\
			\text{s. t.} \quad &c(x) = 0, 
		\end{aligned}
	\end{equation}
	where $\X$ is a closed and convex subset of $\Rn$. Moreover, the objective function $f: \Rn \to \bb{R}$ and the constraint mapping $c: \Rn \to \Rp$ (both possibly nonconvex and nonsmooth) are assumed to be locally Lipschitz continuous over $\Rn$. Throughout this paper, we are particularly interested in the scenarios where the objective function $f$ and constraints $c$ take expectation formulations, implying that only noisy evaluations of function values and subdifferentials of $f$ and $c$ are available in solving \eqref{Prob_Ori}.

	The Lagrangian-based approach is an important class of penalty methods \cite{xie2021complexity} for solving the constrained optimization problem \eqref{Prob_Ori}. Tracing back to the classical works of Hestenes \cite{hestenes1969multiplier}  and Powell \cite{powell1969method}, Lagrangian-based methods are commonly employed to solve large-scale constrained optimization problems. Let the augmented Lagrangian penalty function $L_{\rho}(x, \lambda)$ be defined as 
	\begin{equation*}
		L_{\rho} (x, \lambda) := f(x) + \inner{\lambda, c(x)} + \frac{\rho}{2} \norm{c(x)}^2, 
	\end{equation*}
	then as shown in \cite{xie2021complexity}, a popular framework of Lagrangian-based methods can be expressed as, 
	\begin{equation}
		\label{Eq_Intro_vanilla_ALM}
		\tag{ALM}
		\begin{aligned}
			& \xkp \approx \mathop{\arg\min}_{x \in \X} L_{\rho}(x, \lambda_k),  && \quad\text{(primal subproblem)}\\
			& \lambda_{k+1} = \lambda_k + \theta_kc(\xkp).  && \quad \text{(dual ascent)}
		\end{aligned}
	\end{equation}
	The existing works on the convergence of Lagrangian-based methods mainly focus on the cases where $f$ is convex. For the settings where both $f$ and $c$ are nonconvex and nonsmooth, as demonstrated in \cite{cohen2022dynamic}, the understanding of the convergence behavior of these methods remains limited. When we further assume the differentiability of $f$ and $c$ in \eqref{Prob_Ori}, some existing works \cite{andreani2008augmented,birgin2014practical,birgin2018augmented,curtis2015adaptive,sabach2019lagrangian,sahin2019inexact,sabach2022faster,alacaoglu2023complexity,papadimitriou2023stochastic,tang2023self,wang2023strong,tang2024solving} have investigated the convergence of Lagrangian-based methods as in \eqref{Eq_Intro_vanilla_ALM}, where the primal subproblems are required to be solved with controlled accuracy.

    Very recently, some existing works focus on developing linearized Lagrangian-based methods  \cite{bolte2018nonconvex,cohen2022dynamic,zhu2023first,bourkhissi2023complexity,greenstein2023augmented,hallak2023adaptive} for solving \eqref{Prob_Ori}.  These methods update the primal variables $\{\xk\}$ by applying a single proximal gradient step to $L_{\rho}(\cdot, \lambda_k)$, while enjoying global convergence under appropriate conditions.  As their analysis relies on the weak convexity of the objective function $f$ and the constraints $c$ in \eqref{Prob_Ori}, their results cannot be extended to the cases where both $f$ and $c$ are nonconvex nonsmooth, particularly in the absence of Clarke regularity.  To the best of our knowledge,  no existing work has investigated linearized Lagrangian-based methods for solving \eqref{Prob_Ori} with non-Clarke-regular objectives and constraints.

	\paragraph{Existing works on nonsmooth optimization}
	In nonsmooth optimization, it has been discussed in \cite{daniilidis2020pathological} that general Lipschitz continuous functions can manifest highly pathological properties, and subgradient methods may fail to find any critical point for $f$. Consequently, the majority of existing works restrict their analysis to the case where the objective function $f$ admits a conservative subdifferential (also called conservative field), in the sense that there exists a locally bounded graph-closed set-valued mapping $\D_f$ such that for any absolutely continuous mapping $\gamma: [0, +\infty) \to \Rn$, it holds that 
	\begin{equation*}
		f(\gamma(1)) - f(\gamma(0)) = \int_{0}^{1} \sup_{l_s \in \D_f(\gamma(s))} \inner{l_s, \dot{\gamma}(s)} \mathrm{d}s.
	\end{equation*}
	The extension of conservative field to mappings over $\Rn$ can be found in \cite[Definition 4]{bolte2021conservative}, where such a set-valued mapping is referred to as the {\it conservative Jacobian}. With the concept of conservative field and conservative Jacobian, the class of mappings that admit conservative Jacobians is termed {\it path-differentiable} mappings \cite{bolte2021conservative}.  
	As discussed in \cite{davis2020stochastic,bolte2021conservative}, the class of path-differentiable mappings encompasses all functions whose graphs are definable in an $o$-minimal structure (i.e., definable functions). Hence path-differentiable mappings are general enough to cover a wide range of objective functions and constraints in real-world applications, especially those related to neural networks training tasks \cite{davis2020stochastic,bolte2021conservative,castera2021inertial}.   More importantly, \cite{bolte2020mathematical,bolte2021conservative} show that when a mapping is formulated as the composition of definable functions, the outputs of an automatic differentiation algorithm (AD) for differentiating such a mapping are enclosed in its corresponding conservative Jacobian. Therefore, the concepts of conservative field and conservative Jacobian are capable of characterizing the outputs of AD algorithms, which are implemented in training nonsmooth neural networks in practice.  
	
	Towards the convergence of stochastic subgradient methods in minimizing nonsmooth path-differentiable function $h$ over $\X$, i.e., 
	\begin{equation}
		\label{Prob_simple_min_X}
		\min_{x \in \X} ~ h(x), 
	\end{equation}
	the ordinary differential equation (ODE) approach \cite{benaim2005stochastic,borkar2009stochastic,duchi2018stochastic,davis2020stochastic} is a powerful tool for analyzing these stochastic subgradient methods. In the ODE approach, the convergence of the iterates of a stochastic subgradient method follows the convergence of the trajectories of its corresponding noiseless continuous-time differential inclusion. Then based on the ODE approaches, various existing works \cite{davis2020stochastic,ruszczynski2020convergence,bianchi2021closed,bolte2021conservative,castera2021inertial,hu2022improved,le2023nonsmooth}  establish the convergence properties of the stochastic subgradient method for solving \eqref{Prob_simple_min_X} with $\X = \Rn$. For the scenarios with a general convex set $\X$, existing works \cite{davis2020stochastic,ruszczynski2020convergence} mainly focus on the proximal extensions of stochastic subgradient descent (SGD) \cite{davis2020stochastic}, and stochastic subgradient descent with momentum (SGDM) \cite{ruszczynski2020convergence}. However, to our best knowledge, there is no existing work discussing how to develop stochastic Lagrangian-based methods for solving \eqref{Prob_Ori} with the additional nonsmooth nonconvex constraint $c(x)=0$.

	\paragraph{Embedding stochastic subgradient methods into Lagrangian-based methods}
	In the presence of existing rich results on the convergence properties of stochastic subgradient methods for solving \eqref{Prob_simple_min_X}, it is of great importance to investigate how to \textit{embed} these methods into Lagrangian-based methods for solving \eqref{Prob_Ori}. Here, we use the word ``embed'' to emphasize that these stochastic subgradient methods are employed as black-box methods for updating the primal variables $\{\xk\}$, and the global convergence of these linearized Lagrangian-based methods could automatically inherit from the existing results on the global convergence of these embedded stochastic subgradient methods in solving \eqref{Prob_simple_min_X}. Specifically, we focus on a class of stochastic subgradient methods for solving \eqref{Prob_simple_min_X} that take the following update scheme, 
	\begin{equation}
		\label{Eq_Intro_subgradient_method}
		(\xkp, \ykp) \in \Phi_k(g_k, \xk, \yk, \eta_k).
	\end{equation}
	Here $\yk \in \Rq$ refers to the auxiliary variables employed in the selected stochastic subgradient method, such as the heavy-ball momentum \cite{polyak1964some}, the second-order moment estimator \cite{kingma2014adam}, etc. Moreover, $g_k \in \Rn$ represents the stochastic subgradients of $h$ at $\xk$, while $\eta_k > 0$ refers to the stepsizes. Furthermore, the set-valued mapping $\Phi_k: \Rn \times \X \times \Rq \times \bb{R} \rightrightarrows \X \times \Rq$ characterizes the update scheme of a specific stochastic subgradient method. Readers could refer to Section 4 for detailed illustrations of the stochastic subgradient methods that take the form of \eqref{Eq_Intro_subgradient_method}.

	To embed stochastic subgradient methods into Lagrangian-based methods, some existing works \cite{bolte2018nonconvex,li2023stochastic} choose to apply these stochastic subgradient methods to solve the primal subproblems in \eqref{Eq_Intro_vanilla_ALM} to achieve a solution of controlled accuracy.  For example, the framework developed in \cite{bolte2018nonconvex} requires sufficient reductions on the function value of $L_{\rho}(\cdot, \lambda_k)$ in each iteration \cite[Definition 3.1]{bolte2018nonconvex}. However, for nonsmooth nonconvex $f$ and $c$, it is usually challenging to verify the optimality of the yielded solutions. Consequently, these existing works require the weak convexity of both $f$ and $c$, and thus are incapable of dealing with non-Clarke-regular scenarios.  On the other hand, the existing linearized Lagrangian-based methods \cite{cohen2022dynamic,zhu2023first,bourkhissi2023complexity,hallak2023adaptive} only choose to take a proximal gradient descent step to update the primal variables $\{\xk\}$. Given these efficient stochastic subgradient methods that are designed to solve \eqref{Prob_simple_min_X}, how to embed these methods into Lagrangian-based methods for solving \eqref{Prob_Ori} remains unexplored. Therefore, we are driven to study the following question:
	\begin{quote}
            \it 
		Can we embed stochastic subgradient methods into Lagrangian-based methods for solving \eqref{Prob_Ori} with guaranteed convergence, especially for nonsmooth path-differentiable objective functions $f$ and constraints $c$?
	\end{quote}

	\paragraph{Motivations}
	
	For a class of nonconvex equality-constrained optimization problems, the pioneering works by \cite{bolte2018nonconvex,cohen2022dynamic} have introduced frameworks for developing linearized Lagrangian-based methods. In each iteration, these frameworks typically perform a proximal gradient descent step applied to the primal subproblem in \eqref{Eq_Intro_vanilla_ALM}, followed by a dual ascent step for updating the dual variables ${\lambda_k}$. Under the additional assumption of differentiability for both $f$ and $c$, a straightforward application of these frameworks to \eqref{Prob_Ori} yields the following update scheme:
	\begin{equation}
		\label{Eq_Intro_PLALM}
		\left\{
		\begin{aligned}
			& g_k = \nabla f(\xk)+ \xi_{f,k+1},\\
			&J_{c,k}  = [\nabla c_1(\xk), \nabla c_2(\xk),..., \nabla c_p(\xk)],\\
			& l_{x,k} = g_k + J_{c,k} (\lambda_k + \rho  c(\xk)),\\
			& \xkp = \prox_{\X}(\xk - \eta_k l_{x,k}),\\
			& \lambda_{k+1} = \lambda_k + \theta_kc(\xkp). 
		\end{aligned}
		\right.
	\end{equation}
	Here $\xi_{f, k+1}$ characterizes the noises in the evaluation of $\nabla f(\xk)$, and thus $g_k$ refers to the stochastic gradient of $f$ at $\xk$. Moreover, $J_{c,k}$ refers to the Jacobian of the constraint mapping $c$ at $\xk$. Therefore, $l_{x,k}$ is a noisy evaluation of $\nabla_x L_{\rho}(\xk, \lambda_{k})$. However, when applied to nonsmooth path-differentiable objectives and constraints, the convergence properties of \eqref{Eq_Intro_PLALM} remain unexplored.

	Furthermore, it is worth mentioning that the update scheme in \eqref{Eq_Intro_PLALM} can be regarded as a stochastic gradient descent-ascent method for the following nonconvex-concave min-max optimization problem, 
	\begin{equation}
		\label{Eq_Intro_minmax}
		\min_{x \in \Rn} \max_{\lambda \in \Rp} ~ L_{\rho}(x, \lambda),
	\end{equation}
	in the sense that \eqref{Eq_Intro_PLALM} is equivalent to the following update scheme,
	\begin{equation}
		\xkp = \prox_{\X}(\xk - \eta_k \nabla_x L_{\rho}(\xk, \lambda_k)), \quad \lambda_{k+1} = \lambda_k + \theta_k \nabla_{\lambda}  L_{\rho}(\xkp, \lambda_k).
	\end{equation}
	Consequently, some established results \cite{xu2023unified,cohen2024alternating,yang2024data} on nonconvex-concave min-max problems could potentially be applied to solve \eqref{Prob_Ori} through the min-max optimization problem \eqref{Eq_Intro_minmax}. However, these existing works rely on the uniform boundedness of $L_{\rho}$ over the region $\X \times \Rn$, which usually cannot be satisfied in \eqref{Prob_Ori} as $L_{\rho}(x, \lambda)$ is linear with respect to $\lambda$ for any fixed $x$. More importantly, the proofs in most existing works rely on the weak convexity of $f$ and $c$. Therefore, these existing results cannot be extended to solve \eqref{Prob_Ori} when $f$ and $c$ are only assumed to be path-differentiable and may lack Clarke regularity. 
	
	To overcome the aforementioned difficulties, we introduce the following modified augmented Lagrangian penalty function, which incorporates an additional term $-\frac{1}{2\beta} \norm{c(x)}\norm{\lambda}^2$ into $L_{\rho}(x, \lambda)$, 
	\begin{equation}
		\label{Eq_Intro_modified_ALM_function}
		H_{\rho, \beta}(x, \lambda) := f(x) + \inner{\lambda, c(x)} + \frac{\rho}{2} \norm{c(x)}^2 - \frac{1}{2\beta} \norm{c(x)}\norm{\lambda}^2.
	\end{equation}
	Given that $H_{\rho, \beta}(x, \lambda)$ is strongly concave with respect to the dual variables $\lambda$ for any $x \in \X$ satisfying $c(x) \neq 0$, by taking a single gradient ascent step at $(\xkp, \lambda_k)$ with stepsize $s_k$, we arrive at the following update scheme for the dual variables $\lambda_k$, 
	\begin{equation}
		\label{Eq_Intro_dual_ascent_vanilla}
		\lambda_{k+1} = \lambda_k + s_k \nabla_\lambda H_{\rho, \beta}(\xkp, \lambda_k)= \lambda_k + s_k \left( c(\xkp) - \frac{\norm{c(\xkp)}}{\beta}\lambda_k  \right). 
	\end{equation}
	Moreover, as $H_{\rho,\beta}$ is $\frac{\norm{c(\xkp)}}{\beta}$-strongly concave with respect to $\lambda$ at $(\xkp, \lambda_k)$,  we can choose the stepsize in \eqref{Eq_Intro_dual_ascent_vanilla} as $s_k = \frac{\theta_k}{\norm{c(\xkp)}}$ for some positive parameters $\{\theta_k\}$.  Therefore, the update scheme \eqref{Eq_Intro_dual_ascent_vanilla} can be reshaped as 
	\begin{equation}
		\label{Eq_Intro_dual_ascent}
		\lambda_{k+1} = \lambda_k + \theta_k \left( \mathrm{regu}(c(\xkp)) - \frac{1}{\beta}\lambda_k  \right), \quad \text{where}\quad 
		\mathrm{regu}(y) := 
		\begin{cases}
			0, & y = 0,\\
			\frac{y}{\norm{y}}, & y \neq 0. 
		\end{cases}
	\end{equation}
	Similar choices of the stepsizes for updating the dual variables $\lambda_k$ in Lagrangian-based methods can be found at \cite{li2021rate,li2023stochastic}, where the dual variables $\{\lambda_k\}$ are updated by 
	\begin{equation}
		\label{Eq_Intro_dual_ascent_iALM}
		\lambda_{k+1} = \lambda_k + \min\left\{ \frac{\theta_k}{\norm{c(\xkp)}}, \tilde{\beta} \sigma^k   \right\} \cdot c(\xkp),  
	\end{equation}
	with a predetermined constants $\sigma > 1$ and $\tilde{\beta} > 0$.

	\paragraph{Contributions}
	The contributions of our paper are summarized as follows.
	\begin{itemize}
		\item {\bf  A framework for embedding stochastic subgradient methods}\\
		In this paper, we introduce a framework for embedding stochastic subgradient methods \eqref{Eq_Intro_subgradient_method} into Lagrangian-based methods as follows:
		\begin{equation}
			\label{Eq_Framework}
			\tag{ELM}
			\left\{
			\begin{aligned}
				& l_{x,k} = d_k  + J_{c,k} (\lambda_k +  \rho \wk) + \xi_{k+1},\\
				& (\xkp, \ykp) \in \Phi_k(l_{x,k}, \xk, \yk, \eta_k),\\
				& \text{Compute $\wkp$ as an approximated evaluation to $c(\xkp)$},\\
				& \lambda_{k+1} = \lambda_k + \theta_k \left( \mathrm{regu}(\wkp) - \frac{1}{\beta}\lambda_k  \right).  
			\end{aligned}
			\right.
		\end{equation}
		Let $\D_f$ and $\D_c$ be the conservative Jacobians for $f$ and $c$, respectively. In the above, $d_k \in \Rn$ and $J_{c, k} \in \Rnp$ are inexact noiseless evaluations of $\D_f(\xk)$ and $\D_c(\xk)$ respectively. As a result, $l_{x,k} \in \Rn$ can be regarded as a noisy evaluation for the subgradient of $L_{\rho}(\cdot, \lambda_k)$ at $\xk$, where its evaluation noise is characterized by $\xi_{k+1}$. Moreover, $\beta > 0$ and $\rho \geq 0$ are two parameters for the framework. Different from existing Lagrangian-based methods, in each iteration, \eqref{Eq_Framework} updates its primal variables $\{(\xk, \yk)\}$ by using a single step of the subgradient method in \eqref{Eq_Intro_subgradient_method}, while the dual variables $\{\lambda_k\}$ are updated by a dual-ascent step for the modified augmented Lagrangian penalty function in \eqref{Eq_Intro_modified_ALM_function}.

		\item {\bf Global convergence of \eqref{Eq_Framework}}

		Under mild assumptions, we prove that $l_{x,k}$ asymptotically approximates a noisy evaluation of $\D_g(\xk)$, where $\D_g: \Rn \rightrightarrows \Rn$ is a conservative field for the following penalty function:
		\begin{equation}
			\label{Eq_Intro_penalty_function}
			g(x) := f(x) + \beta \norm{c(x)} + \frac{\rho}{2} \norm{c(x)}^2. 
		\end{equation}
		Then based on the convergence properties of the embedded stochastic subgradient method, we prove that any cluster points of the sequence $\{\xk\}$ are contained in the set $\{x \in \X: 0 \in \D_g(x) + \NX(x)\}$ under mild conditions. Moreover, we prove that $g$ is an exact penalty function of \eqref{Prob_Ori} under certain regularity conditions. Hence any cluster point of $\{\xk\}$ is a KKT point of \eqref{Prob_Ori} in the sense of the conservative field.  
		Furthermore,  we show that our framework \eqref{Eq_Framework} can be applied to design Lagrangian-based methods for nonsmooth optimization problems with nonsmooth expectation constraints (i.e., the constraint mapping $c$ in \eqref{Prob_Ori} takes a stochastic formulation).

		\item {\bf Algorithmic consequence}
		
		We revisit the convergence of proximal SGD and momentum SGD, and also establish the convergence of proximal ADAM on solving \eqref{Prob_simple_min_X} based on ODE approaches. Then we show that a wide range of stochastic subgradient methods, including the proximal SGD, proximal momentum SGD, and proximal ADAM, can be embedded into Lagrangian-based methods through our proposed framework \eqref{Eq_Framework}. Preliminary numerical experiments demonstrate that these methods exhibit comparable performance as existing Lagrangian-based methods when applied to training deep neural networks with constraints. These numerical results further demonstrate that our proposed framework \eqref{Eq_Framework} enables the direct implementation of various efficient proximal subgradient methods for \eqref{Prob_Ori}, while benefiting from their high efficiency with guaranteed global convergence.

	\end{itemize}
	
	In summary, our proposed framework \eqref{Eq_Framework} provides a shortcut for directly applying existing stochastic subgradient methods, which are designed to solve \eqref{Prob_simple_min_X}, into the minimization of the constrained optimization problem \eqref{Prob_Ori}. More importantly, the convergence properties of the framework \eqref{Eq_Framework} directly follow from the rich existing results on the employed stochastic subgradient methods under mild conditions. We present a brief illustration of our results in Figure \ref{Figure_embed}. 
	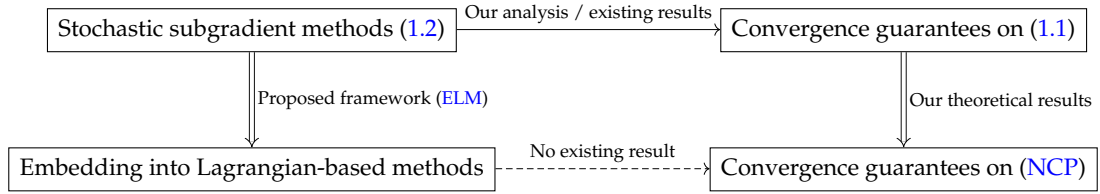
\begin{figure}[!htbp]
            \small
		\caption{A brief illustration of our results on embedding stochastic subgradient methods, which are designed to solve \eqref{Prob_simple_min_X}, into linearized Lagrangian-based methods through the framework \eqref{Eq_Framework}.}
		\label{Figure_embed}
		\begin{equation*}
			\begin{tikzcd}[row sep=4em,column sep=9em,cells={nodes={draw=black}}]
				\text{Stochastic subgradient methods \eqref{Eq_Intro_subgradient_method}} \arrow[d, Rightarrow, "\text{Proposed framework \eqref{Eq_Framework}}"]  \arrow[r, rightarrow, "\text{Our analysis / existing results}"] & \text{Convergence guarantees on \eqref{Prob_simple_min_X}} \arrow[d, Rightarrow, "\text{Our theoretical results}"]\\
				\text{Embedding into Lagrangian-based methods}  \arrow[r, dashrightarrow, "\text{No existing result}"] & \text{Convergence guarantees on \eqref{Prob_Ori}}
			\end{tikzcd}
		\end{equation*}
	\end{figure}

\paragraph{Organization}
The rest of this paper is organized as follows. In Section 2, we provide an overview of the notations used throughout the paper and present the necessary preliminary concepts related to nonsmooth analysis and differential inclusion. In Section 3, we present the convergence analysis of the proposed framework \eqref{Eq_Framework}.  Section 4  presents illustrative examples of how to embed stochastic
subgradient methods into Lagrangian-based methods using the framework \eqref{Eq_Framework}. In Section 5, we present the results of our numerical experiments that investigate the performance of our proposed Lagrangian-based methods for solving nonsmooth constrained optimization problems. Finally, we conclude the paper in the last section.

	\section{Preliminaries}

	\subsection{Notations}
	For any vectors $x$ and $y$ in $\Rn$ and $\delta \in \bb{R}$, we denote $x\odot y$, $x^{ \delta}$, $x/y$, $|x|$, $x+\delta$, $\sqrt{x}$ as the vectors whose $i$-th entries are given by $x_iy_i$, $x_i^{\delta}$, $x_i/y_i$, $|x_i|$, $x_i + \delta$, and $\sqrt{x_i}$, respectively. 
	We denote $\Rn_+:=\{ x\in \Rn: x_i\geq 0 \text{ for any } 1\leq i\leq n \}$. Moreover, for any subsets $\ca{X}, \ca{Y} \subset \Rn$, we denote $\ca{X}\odot \ca{Y}:= \{x\odot y: x \in \ca{X}, y\in \ca{Y} \}$,  $|\ca{X}|:= \{|x|: x \in \ca{X}\}$ and $\norm{\ca{X}} := \sup\{ \norm{w} : w\in \ca{X}\}$. For any $z \in \Rn$, we denote $z + \ca{X} := \{z\} + \ca{X}$ and $z \odot \ca{X} := \{z\} \odot\ca{X}$. In addition, for any $\ca{W} \subseteq \bb{R}^{m\times n}$ any $\y \subset \Rn$ and any $z \in \Rn$, we denote $\ca{W} z := \{Wz: W \in \ca{W}\}$ and $\ca{W}\ca{Y} := \{Wy: W \in \ca{W}, y \in \ca{Y}\}$.  For a closed convex set $\X \subset \Rn$, we denote $\NX(x)$ as the normal cone of $\X$ at $x \in \X$, and denote $\prox_{\X}(x):= \argmin_{y \in \X} \norm{y-x}^2$. 
 
	
	For any subset $\ca{X} \subseteq \Rn$, we denote $\widetilde{\mathrm{span}}(\ca{X})$ as the smallest subspace containing $\ca{X}$. Moreover, for any $\ca{Z} \subseteq \Rnp$, we denote $\mathrm{span}(\ca{Z}) := \widetilde{\mathrm{span}}(\{Z\lambda: Z \in \ca{Z}, \lambda \in \Rp\})$.

	Furthermore, for any positive sequence $\{\theta_k\}$, we define 
	$\lambda_0 := 0$, $\lambda_i := \sum_{k = 0}^{i-1} \theta_k$ for $i\geq 1$, and $\Lambda(t) := \sup  \{k \geq 0: t\geq \lambda_k\} $. More explicitly, $\Lambda(t) = p$ if $\lambda_p \leq t < \lambda_{p+1}$ for any $p \geq 0$. In particular, $\Lambda(\lambda_p) = p.$ In addition, we denote $\Tregu: \Rp \rightrightarrows \Rp$ as 
	\begin{equation}
		\Tregu(y) := 
		\begin{cases}
			\{y \in \Rp: \norm{y} \leq 1\}, & y = 0,\\
			\left\{\frac{y}{\norm{y}}\right\}, & y \neq 0. 
		\end{cases}
	\end{equation}

	We denote $(\Omega, \ca{F}, \mathbb{P})$ as the probability space that captures the randomness of our analyzed stochastic algorithms. We say that $\{\ca{F}_k\}_{k \in \bb{N}}$ is a filtration  if  $\{\ca{F}_k\}$ is a sequence of $\sigma$-algebras that satisfies $ \ca{F}_0 \subseteq \ca{F}_1 \subseteq \cdots \subseteq \ca{F}_{\infty} \subseteq \ca{F}$. Moreover, we say that a sequence of random vectors $\{\xi_k\}$ is a martingale difference sequence if  $\{\xi_k\}$ is adapted to the filtration $\{ \ca{F}_{k} \}$, $\bb{E}[|\xi_k|] < \infty$ and $\bb{E}\left[ \xi_{k+1} | \ca{F}_{k} \right] = 0$ holds almost surely for any $k\geq 0$. 
A martingale difference sequence $\{\xi_k\}$ is said to be uniformly bounded if there exists a constant $M_{\xi}$ such that $\sup_{k\geq 0}\norm{\xi_k} \leq  M_{\xi}$.

	\subsection{Nonsmooth analysis}
	\label{Section_Nonsmooth_Analysis}

	In this subsection, we introduce some basic concepts in nonsmooth optimization, especially those related to the concept of conservative field \cite{bolte2021conservative}. Interested readers could refer to \cite{bolte2021conservative,davis2020stochastic} for more details. 
	
	We begin our introduction on the concept of Clarke subdifferential \cite{clarke1990optimization}, which plays an essential role in characterizing stationarity and the development of algorithms for nonsmooth optimization problems.

	\begin{defin}[\cite{clarke1990optimization}]
		\label{Defin_Subdifferential}
		For any given locally Lipschitz continuous function $f: \Rn \to \bb{R}$ and any $x \in \Rn$, 
		the Clarke subdifferential $\partial f$ is defined as 
            \begin{equation}
                \partial f(x) := \mathrm{conv}\left( \left\{ \lim_{k \to +\infty} \nabla f(\xk): \xk \to x \text{ and $f$ is differentiable at $\{\xk\}$} \right\}  \right). 
            \end{equation}
	\end{defin}

	Next, we present a brief introduction to the concept of conservative field, which can be applied to characterize how nonsmooth neural networks are differentiated by automatic differentiation (AD) algorithms.
	
	\begin{defin}
		A set-valued mapping $\ca{D}: \Rn \rightrightarrows \Rs$ is a mapping from $\Rn$ to a collection of subsets of $\bb{R}^s$. $\D$ is said to have closed graph, or is graph-closed, if the graph of $\ca{D}$ defined by
		\begin{equation*}
			\mathrm{graph}(\D) := \left\{ (w,z) \in \Rn \times \Rs: w \in \bb{R}^n, z \in \D(w) \right\},
		\end{equation*}
		is a closed subset of $\Rn \times \Rs$.  
	\end{defin}
	
	\begin{defin}
		A set-valued mapping $\ca{D}: \Rn \rightrightarrows \Rs$ is said to be locally bounded if, for any $x \in \Rn$, there exists a neighborhood $V_x$ of $x$ such that $\cup_{y \in V_x}\ca{D}(y)$ is a bounded subset of $\bb{R}^s$. 
	\end{defin}

	\begin{defin}[Aumann’s integral]
		\label{Defin_Aumann_integral}
		Let $(\Theta, \ca{F}, P)$ be a measurable space, and $\D: \Rn  \times \Theta \rightrightarrows \Rn$ be a measurable set-valued mapping. Then for all $x \in \Rn$, the integral of $\ca{D}$ with respect to $P$ is defined as 
		\begin{equation*}
			\bb{E}_{s \sim P}\left[ \D(x, s)  \right] := \left\{ \int_{\Theta} \chi(x, s) ~\mathrm{d}P(s): \text{  $\chi(x, \cdot)$  is integrable, and $\chi(x, s) \in \D(x, s)$ for any $s \in \Theta$}  \right\}. 
		\end{equation*}
	\end{defin}

	\begin{defin}
		An absolutely continuous curve is a continuous mapping $\gamma: \bb{R}_+ \to \Rn $ whose derivative $\gamma'$ exists almost everywhere in $\bb{R}_+$ and $\gamma(t) - \gamma(0)$ equals the Lebesgue integral of $\gamma'$ between $0$ and $t$ for all $t \in \bb{R}_+$, i.e.,
		\begin{equation*}
			\gamma(t) = \gamma(0) + \int_{0}^t \gamma'(u) \mathrm{d} u, \qquad \text{for all $t \in \bb{R}_+$}.
		\end{equation*}
	\end{defin}

	\begin{defin}
		\label{Defin_conservative_jacobian}
		A conservative Jacobian for a locally Lipschitz mapping $H: \Rn \to \Rs$ is a set-valued mapping $\D_H: \Rn \rightrightarrows \bb{R}^{n\times s}$, which is nonempty-valued, locally bounded, graph-closed, and for any absolutely continuous curve $\gamma: \bb{R}_+ \to \Rn$, it holds for almost every $t \in \bb{R}_+$ that 
		\begin{equation}
			\frac{\mathrm{d}(H\circ \gamma) }{\mathrm{d}t} (t) = J^\top \dot{\gamma}(t), \quad \forall\;  J \in \D_H(\gamma(t)).
		\end{equation}
		In particular, when $H$ is a locally Lipschitz function (i.e., $s = 1$), the conservative Jacobian for $H$ is also referred to as the conservative field for $H$. 
	\end{defin}
	
	A mapping $H: \Rn \to \Rs$ is called path-differentiable if it is locally Lipschitz continuous and admits a conservative Jacobian $\D_H$. As demonstrated in \cite{bolte2021conservative}, $H$ is path-differentiable if and only if its Clarke Jacobian is a conservative Jacobian for $H$.  In addition,  it is important to note that any conservative Jacobian is locally bounded  \cite[Remark 3]{bolte2021conservative}.

	It is worth mentioning that the class of path-differentiable functions is general enough to cover the objectives in a wide range of real-world problems. As shown in \cite[Section 5.1]{davis2020stochastic}, any Clarke regular function is path-differentiable. Beyond Clarke regular functions, another important class of path-differentiable functions are functions whose graphs are definable in an $o$-minimal structure \cite[Definition 5.10]{davis2020stochastic}. Usually, the $o$-minimal structure is fixed, and we simply call these functions definable. As demonstrated in \cite{van1996geometric}, any definable function admits a Whitney $C^s$ stratification \cite[Definition 5.6]{davis2020stochastic} for any $s \geq 1$, and hence is path-differentiable  \cite{bolte2021conservative,davis2020stochastic}. To characterize the class of definable functions, \cite{davis2020stochastic,bolte2021conservative,bolte2022differentiating} shows that numerous common activation functions and dissimilarity functions are all definable. Furthermore, since definability is preserved under finite summation and composition  \cite{bolte2021conservative,davis2020stochastic},  for any neural network built from definable blocks, its loss function is definable and thus belongs to the class of path-differentiable functions.

	Moreover, \cite{bolte2007clarke} shows that any Clarke subdifferential of definable functions is definable. Consequently, for any neural network constructed from definable blocks, the conservative field corresponding to the AD algorithms can be chosen as a definable set-valued mapping formulated by compositing the Clarke subdifferentials of all its building blocks \cite{bolte2021conservative}.

	As demonstrated in \cite{bolte2021conservative}, the concept of conservative Jacobian can be regarded as a generalization of Clarke subdifferential, hence can be applied to characterize the stationarity for nonsmooth optimization problems. 
	For \eqref{Prob_simple_min_X}, we present the definition of its stationary points based on the concept of conservative field. 
	\begin{defin}
		Given any path-differentiable function $h: \Rn \to \bb{R}$, and suppose $h$ admits $\D_h$  as its conservative field. Then for any $x \in \X$, we say $x$ is a $\D_h$-stationary point of \eqref{Prob_simple_min_X} if 
		\begin{equation}
			0 \in \conv\left( \D_h(x) + \ca{N}_{\X}(x) \right). 
		\end{equation}
	\end{defin}
	Then based on \cite{bolte2007clarke,bolte2021conservative}, we present the following proposition to show that the definability of $h$ and $\mathcal{D}_h$ leads to the nonsmooth Morse–Sard property \cite{bolte2007clarke} for \eqref{Prob_simple_min_X}. 
	\begin{prop}[Corollary 5 in \cite{bolte2007clarke}]
		\label{Prop_definable_regularity}
		Let $h: \Rn \to \bb{R}$ be a path-differentiable function that admits $\D_h$ as its conservative field.  Suppose $h$, $\D_h$ and $\X$ are definable over $\Rn$, then $\{h(x): x \in \X,~ 0\in \conv(\D_h(x)) + \ca{N}_{\X}(x)\}$ is a finite subset of $\bb{R}$. 
	\end{prop}
	\begin{proof}
		From the definability of $h$ and $\X$, \cite{bolte2007clarke} demonstrates that there exists a stratification $\{\M_{i}\}$ such that 
		\begin{equation}
			\ca{P}_{T_x \M_{i(x)}} (\partial h(x) + \ca{N}_{\X}(x)) = \{\grad h(x)\}.
		\end{equation}
            Here $\M_{i(x)}$ refers to the underlying manifold of $x$,   $T_x \M_{i(x)}$ refers to the tangent space at $x$ of the manifold $\ca{M}_i$, $\mathcal{P}$ refers to the orthogonal projection, and $\grad h(x)$ refers to the Riemannian gradient of $f$ with respect to $\M_{i(x)}$.
  
		Together with \cite{bolte2021conservative}, the sets 
		\begin{equation}
			R_i = \{x \in \M_i: \grad h(x)  \neq \ca{P}_{T_x \M_{i(x)}} (\D_h(x) + \ca{N}_{\X}(x))\}, \quad i = 1,.., q,
		\end{equation}
		have dimensions that are strictly lower than that of $\M_i$. Then based on the same techniques as \cite[Theorem 4]{bolte2021conservative}, we can progressively refine the stratification to find a refined stratification $\{\tilde{\M}_i\}$ such that 
		\begin{equation}
			\ca{P}_{T_x \tilde{\M}_{i(x)}} (\D_h(x) + \ca{N}_{\X}(x)) = \{\grad h(x)\}.
		\end{equation}
		Applying the definable Sard's theorem \cite{davis2020stochastic,bolte2021conservative} to each manifold $\M_i$, the results follow. 
	  \end{proof}

	For the optimization problem \eqref{Prob_Ori}, we present the following 
	definition of its KKT points based on the concept of conservative Jacobian. 
	\begin{defin}
		\label{Defin_KKT}
		Suppose  both $f: \Rn \to \bb{R}$ and $c: \Rn \to \Rp$ 
	 in \eqref{Prob_Ori} 	are path-differentiable, and admit $\D_f$ and $\D_c$ as their conservative Jacobian, respectively. Then for any $x \in \X$, we say that $x$ is a $(\D_f, \D_c)$-KKT point of \eqref{Prob_Ori} if $c(x) = 0$ and 
		\begin{equation}
			0 \in \D_f(x) +  \mathrm{span}( \D_c(x) )  + \ca{N}_{\X}(x). 
		\end{equation}
	\end{defin}
	
	It is worth mentioning that when both $f$ and $c$ are differentiable, Definition \ref{Defin_KKT} coincides with the KKT conditions for constrained optimization \cite{jorge2006numerical,beck2014introduction}. 
	
	\subsection{Differential inclusion and stochastic subgradient methods}
	In this subsection, we introduce some fundamental concepts related to the stochastic approximation technique that are essential for the proofs presented in this paper. The concepts discussed here are mainly adopted from \cite{benaim2005stochastic}. Interested readers could refer to \cite{benaim2006dynamics,benaim2005stochastic,borkar2009stochastic,davis2020stochastic} for more details on the stochastic approximation technique. 
	
	We start by presenting the definition of the trajectory of a differential inclusion in the following.
	\begin{defin}
		For any graph-closed set-valued mapping $\D: \Rn \rightrightarrows \Rn$,  we say that an absolutely continuous path $x(t)$ in $\Rn$ is a trajectory of the differential inclusion 
		\begin{equation}
			\label{Eq_def_DI}
			\frac{\mathrm{d} x}{\mathrm{d}t} \in \D(x),
		\end{equation}
		with initial point $x_0$ if $x(0) = x_0$, and $\dot{x}(t) \in \D(x(t))$ holds for almost every $t\geq 0$. 
	\end{defin}
	It is worth mentioning that \cite{aubin2012differential} has established
	that the trajectory of \eqref{Eq_def_DI} exists whenever $\D$ is a locally bounded graph-closed set-valued mapping.

	\begin{defin}
		\label{Defin_delta_expansion}
		For any given set-valued mapping $\D: \Rn \rightrightarrows \Rn$ and any constant $\delta \geq 0$,  the $\delta$-expansion of $\D$, denoted as $\D^{\delta}$,  is defined as
		\begin{equation}
			\D^{\delta}(x) := \{ w \in \Rn: \exists z \in \bb{B}_{\delta}(x), \, \mathrm{dist}(w, \ca{D}(z))\leq \delta \}.
		\end{equation}
	\end{defin}

	\begin{defin}
		\label{Defin_Lyapunov_function}
		Let $\ca{B} \subset \Rn$ be a closed set. A continuous function $\phi:\Rn \to \bb{R}$ is referred to as a Lyapunov function for the differential inclusion \eqref{Eq_def_DI}, with the stable set $\ca{B}$, if it satisfies the following conditions.
		\begin{enumerate}
			\item For any $\gamma$ that is a solution for \eqref{Eq_def_DI} with 
			$\gamma(0) \in \Rn$ , it holds that $\phi(\gamma(t)) \leq \phi(\gamma(0))$ for any $t\geq0$.
			\item For any $\gamma$ that is a solution for \eqref{Eq_def_DI} with $\gamma(0) \notin \ca{B}$, there exists a constant $T>0$ such that {$\phi(\gamma(T)) < \sup_{t\in[0,T]}\phi(\gamma(t)) =\phi(\gamma(0))$}.
		\end{enumerate}
	\end{defin}

	\begin{defin}
		\label{Defin_perturbed_solution}
		We say that an absolutely continuous function $\gamma$
		is a perturbed solution to \eqref{Eq_def_DI}  if there exists a locally integrable function $u: \bb{R}_+ \to \Rn$, such that 
		\begin{itemize}
			\item For any $T>0$, it holds that $\lim\limits_{t \to \infty} \sup\limits_{0\leq l\leq T} \norm{\int_{t}^{t+l} u(s) ~\mathrm{d}s} = 0$. 
			\item There exists $\delta: \bb{R}_+ \to \bb{R}$ such that $\lim\limits_{t \to \infty} \delta(t) = 0$ and $\dot{\gamma}(t) - u(t) \in \D^{\delta(t)}(\gamma(t))$. 
		\end{itemize}
	\end{defin}

	Given a sequence of positive real numbers $\{\eta_k\}$, we then present the definition of the  (continuous-time) interpolated process of $\{\xk\}$ with respect to $\{\eta_k\}$  as follows. 
	\begin{defin}
		For any given sequence of positive real numbers $\{\eta_k\}$, the  (continuous-time) interpolated process of $\{\xk\}$ with respect to $\{\eta_k\}$ is the mapping $w: \bb{R}_+ \to \Rn$ such that 
		\begin{equation}
			w(\lambda_i + s) := x_{i} + \frac{s}{\eta_i} \left( x_{i+1} - x_{i} \right), \quad s\in[0, \eta_i). 
		\end{equation}
		Here $\lambda_0 := 0$, and $\lambda_i := \sum_{k = 0}^{i-1} \eta_k$ for $i\geq 1$.
	\end{defin}

	Now consider the sequence $\{\xk\}$ generated by the  following update scheme,  
	\begin{equation}
		\label{Eq_def_Iter}
		\xkp = \xk + \eta_k(d_k + \xi_{k+1}),
	\end{equation}
	with a diminishing positive sequence of real numbers $\{\eta_k\}$.  The following lemma is an extension of \cite[Proposition 1.3]{benaim2005stochastic}, which allows for inexact evaluations of the set-valued mapping $\D$. It shows that for the sequence $\{\xk\}$ generated by \eqref{Eq_def_Iter}, its interpolated process with respect to $\{\eta_k\}$ is a perturbed solution of the differential inclusion \eqref{Eq_def_DI}.

	\begin{lem}
		\label{Le_interpolated_process}
		Let $\ca{D}: \Rn \rightrightarrows \Rn$ be a set-valued mapping that is nonempty compact convex-valued and graph-closed.
		Suppose the following conditions hold in \eqref{Eq_def_Iter}.
		\begin{enumerate}
			\item For any $T> 0$, it holds that $\lim\limits_{s \to \infty} \sup\limits_{s\leq i \leq \Lambda(\lambda_s + T)}\norm{ \sum_{k = s}^{i} \eta_k \xi_{k+1}} =0$. 
			\item There exist a  nonnegative sequence $\{\delta_k\}$  such that $\lim_{k\to \infty} \delta_k = 0$, and $d_k \in \D^{\delta_k}(\xk)$ holds for any $k\geq 0$.
			\item $\sup_{k \geq 0} \norm{\xk}<\infty$, $\sup_{k \geq 0} \norm{d_k} < \infty$. 
		\end{enumerate}
		Then the interpolated process of $\{\xk\}$ is a perturbed solution for \eqref{Eq_def_DI} with respect to $\{\eta_k\}$. 
	\end{lem}

	The following theorem summarizes the results in \cite{davis2020stochastic}, which 
	establishes the convergence of $\{\xk\}$ generated by \eqref{Eq_def_Iter}. Therefore, we omit its proof for simplicity. 
	
	\begin{theo}
		\label{The_convergence_beniam}
		Let $\ca{D}: \Rn \rightrightarrows \Rn$ be a locally bounded set-valued mapping that is nonempty compact convex-valued and graph-closed. For any sequence $\{\xk\}$, suppose there exists a continuous function $\Psi: \Rn \to \bb{R}$ and a closed subset $\ca{B}$ of $\Rn$ such that
		\begin{enumerate}
			\item $\Psi$ is bounded from below, and the set $\{\Psi(x) \,:\,x\in\ca{B}\}$
			has empty interior in $\bb{R}$.
			\item $\Psi$ is a Lyapunov function for the differential inclusion \eqref{Eq_def_DI} that admits $\ca{B}$ as its stable set. 
			\item The interpolated process of $\{\xk\}$ is a perturbed solution of  \eqref{Eq_def_DI}. 
		\end{enumerate}
		Then any cluster point of $\{\xk\}$ lies in $\ca{B}$, and the sequence $\{\Psi(\xk)\}$ converges. 
	\end{theo}
	
	Similar frameworks under slightly different conditions can be found in \cite{borkar2009stochastic,davis2020stochastic,duchi2018stochastic}, while some recent works \cite{bianchi2021closed,bolte2022long} focus on analyzing the convergence of \eqref{Eq_def_Iter} under more relaxed conditions. Interested readers could refer to those works for details.

\section{Global Convergence}
	In this section, we prove the global convergence of the framework \eqref{Eq_Framework}.
	In Section 3.1, we introduce the class of stochastic subgradient methods that can be embedded into the Lagrangian-based methods through the framework \eqref{Eq_Framework}. Then Section 3.2 demonstrates that under mild conditions, the sequence $\{\xk\}$ generated by 
	the framework \eqref{Eq_Framework} converges to  $(\D_f, \D_c)$-KKT points of \eqref{Prob_Ori}. Additionally, Section 3.3 presents the applicability of our framework  \eqref{Eq_Framework} to \eqref{Prob_Ori} with expectation-formulated constraints.

	\subsection{Embeddable stochastic subgradient methods}
	We begin by introducing the concept of  
	``embeddable stochastic subgradient method'', which refers to a class of subgradient methods characterized by \eqref{Eq_Intro_subgradient_method}. In particular, the embeddable stochastic subgradient methods enjoy convergence guarantees and can be embedded into the Lagrangian-based methods through framework \eqref{Eq_Framework}. We first make the following assumption on the set-valued mappings $\{\Phi_k\}$ in \eqref{Eq_Intro_subgradient_method}. 
	\begin{assumpt}
		\label{Assumption_Phi}
		There exists a locally bounded mapping $T_{\Phi}: \Rn \times \X \times  \Rq \to \bb{R}_+$  such that 
		\begin{equation}
			\label{Eq_condition_Phi}
			\sup_{k\geq 0} \left\{ \mathrm{dist}\left(\Phi_k(g, x, y, \eta), (x, y)\right)  \right\} \leq \eta T_{\Phi}(g,x,y),  \quad \forall\; (g,x,y, \eta) \in  \Rn \times \X \times  \Rq \times \bb{R}_+.
		\end{equation}
	\end{assumpt}
	Assumption \ref{Assumption_Phi} illustrates that $\eta_k$ controls the distance between successive iterates $(\xkp, \ykp)$ and $(\xk, \yk)$ for all $k\geq 0$. Hence the sequence $\{\eta_k\}$ can be regarded as the stepsizes of the stochastic subgradient methods in \eqref{Eq_Intro_subgradient_method}.

	Then following the existing ODE approaches \cite{benaim2005stochastic}, we introduce the following conditions on the optimization problem \eqref{Prob_simple_min_X} and the stochastic subgradient method \eqref{Eq_Intro_subgradient_method}. 
	\begin{cond}
		\label{Cond_stable_alg}
		\begin{enumerate}
			\item There exists a path-differentiable function $h$ that admits $\D_h$ as its convex-valued conservative field. Moreover, the set $\{h(x): 0 \in \D_h(x) + \NX(x), ~x \in \X\}$ has empty interior in $\bb{R}$. 
			\item The sequence of iterates $\{(\xk, \yk)\}$ is uniformly bounded in $\X \times \Rq$.
			\item There exists a nonnegative diminishing sequence $\{\delta_k\} \subset \bb{R}_+$ and a sequence of noises $\{\xi_k\} \subset \Rn$ such that $g_k \in \D_h^{\delta_k}(\xk) + \xi_{k+1}$ holds for all $k\geq 0$. 
			\item For any $T>0$, the stepsizes $\{\eta_k\}$ and the noises $\{\xi_{k}\}$ satisfy 
			\begin{equation}
				\label{Eq_Defin_stable_algorithm_noise}
				\lim_{k\to \infty} \eta_k = 0, \quad \sum_{k = 0}^{\infty} \eta_k = +\infty, \quad \lim_{i\to \infty}\sup_{i\leq j\leq \Lambda(\lambda_i + T)} \norm{\sum_{k = i}^{j} \eta_k \xi_{k+1}} = 0, \quad \sup_{k\geq 0} \norm{\xi_k} <\infty.  
			\end{equation}
		\end{enumerate}
	\end{cond}
	
	The formulation of \eqref{Eq_Intro_subgradient_method} is general enough to encompass a wide range of stochastic subgradient methods that enjoy convergence guarantees from ODE approaches, including proximal SGD \cite{davis2020stochastic},  proximal momentum SGD \cite{ruszczynski2020convergence,le2023nonsmooth}. Moreover, when  $\X = \Rn$, \eqref{Eq_Intro_subgradient_method} encompasses a wide range of stochastic subgradient methods for unconstrained optimization, such as INNA \cite{castera2021inertial}, Adam \cite{xiao2023adam}, SignSGD \cite{xiao2023convergence}, etc.  
	
	These existing works establish the global convergence for their analyzed stochastic subgradient methods based on similar conditions as Condition \ref{Cond_stable_alg}. Typically, for a given stochastic subgradient method expressed in \eqref{Eq_Intro_subgradient_method}, verifying its embeddability usually directly follows from the established proof techniques analogous to those used in demonstrating global convergence through ODE approaches. In the following remark, we make some detailed comments on Condition \ref{Cond_stable_alg}.

	\begin{rmk}
		The Condition \ref{Cond_stable_alg}(1) is referred to as the nonsmooth Morse-Sard's property \cite{bolte2007clarke,davis2020stochastic,bolte2021conservative}, which is an essential condition for establishing the global convergence of the stochastic subgradient methods through ODE approaches. As demonstrated in Proposition \ref{Prop_definable_regularity}, the nonsmooth Morse-Sard's property is implied by the definability of $h$, $\D_h$, and $\X$, hence it is a common assumption in many existing works. Moreover, in ODE approaches, the assumption on the uniform boundedness of the iterates (i.e., $\{(\xk, \yk)\}$ in \eqref{Eq_Intro_subgradient_method}) is usually a default assumption. Interested readers could refer to \cite{davis2020stochastic,bolte2020mathematical,ruszczynski2020convergence,bolte2021conservative,bolte2021nonsmooth,castera2021inertial,bianchi2022convergence,le2023nonsmooth} for more examples on the involvement of nonsmooth Morse-Sard's property and the assumption on the uniform boundedness of the iterates. Therefore, we can conclude that Condition \ref{Cond_stable_alg}(1)-(2) are mild in practice.

		Furthermore, Condition \ref{Cond_stable_alg}(3) illustrates that the sequence $\{\xi_k\}$ can be regarded as the evaluation noises in the evaluation of $\{g_k\}$. Then Condition \ref{Cond_stable_alg}(4) demonstrates that the evaluation noises in the given stochastic subgradient method are controlled. As illustrated in \cite{benaim2005stochastic}, when $\{\xi_k\}$ is a uniformly bounded martingale difference sequence, choosing $\eta_k = o(1/
		log(k))$ guarantees the validity of \eqref{Eq_Defin_stable_algorithm_noise} almost surely. Therefore, Assumption \ref{Cond_stable_alg}(4) is also commonly made in various existing works.

		As a result, we can conclude that all the conditions in  Condition \ref{Cond_stable_alg} are prevalent in existing works on the convergence properties of stochastic subgradient methods by ODE approaches.  
	\end{rmk}

	In the following, we present the definition of the concept of {\it embeddable stochastic subgradient method}, where Condition \ref{Cond_stable_alg} leads to the global convergence of \eqref{Eq_Intro_subgradient_method} for solving \eqref{Prob_simple_min_X}. 
	\begin{defin}
		\label{Defin_stable_stochastic_method}
		Given a sequence of set-valued mappings $\{\Phi_k\}$, we say that $\{\Phi_k\}$ represents an embeddable stochastic subgradient method if Assumption \ref{Assumption_Phi} holds, and for any path-differentiable function $h$, Condition \ref{Cond_stable_alg} guarantees that   any cluster point of the sequence $\{\xk\}$ 
		generated by \eqref{Eq_Intro_subgradient_method} lies in $\{x \in \X, 0 \in \D_h(x) + \NX(x)\}$ and the sequence $\{h(\xk)\}$ converges.
	\end{defin}

	\subsection{Basic assumptions and main results}

	In this subsection, based on the concept of embeddable stochastic subgradient methods in Definition \ref{Defin_stable_stochastic_method}, we present the basic assumptions on the constrained optimization problem \eqref{Prob_Ori} and the framework \eqref{Eq_Framework}. 
	
	We begin our theoretical analysis with the following assumptions on the path-differentiability of the objective function $f$ and constraints $c$. 
	\begin{assumpt}
		\label{Assumption_f}
		\begin{enumerate}
			\item The objective function $f$ is  path-differentiable and admits $\D_f: \Rn \rightrightarrows \Rn$ as its  convex-valued conservative field. 
			\item The constraint mapping $c$ is path-differentiable and admits $\D_{c}: \Rn \rightrightarrows \Rnp$ as its convex-valued conservative Jacobian. 
			\item $\inf_{x \in \X} f(x) >-\infty$. 
		\end{enumerate}
	\end{assumpt}
	
	As discussed in \cite{bolte2021conservative}, the definability of $f$ and $c$ implies their path-differentiability, hence Assumption \ref{Assumption_f} is mild in practice.  Under Assumption \ref{Assumption_f},
 the penalty function $g$ defined in  \eqref{Eq_Intro_penalty_function} is path-differentiable, and  the set-valued mapping $\D_g: \Rn \rightrightarrows \Rn$ defined by
	\begin{equation}
		\D_g(x) := \conv\left(\D_f(x) +  \D_c(x) \left(\beta \Tregu(c(x)) + \rho c(x)\right)\right) 
	\end{equation}
	is a convex-valued conservative field of $g$, as stated in the next lemma. 
	The proof of Lemma \ref{Le_conservative_g} directly follows from the chain rule of conservative field in \cite{bolte2021conservative}, hence is omitted for simplicity. 
	\begin{lem}
		\label{Le_conservative_g}
		Suppose Assumption \ref{Assumption_f} holds, then $g$ is a path-differentiable function that admits $\D_g$ as its convex-valued conservative field. 
	\end{lem}

	Moreover, we present Proposition \ref{Prop_relation_DfDc_Dg} to show that $\{x \in \X: 0 \in \D_g(x) + \ca{N}_{\X}(x), \; c(x) =0\}$ is a subset of the set of $(\D_f, \D_c)$-KKT points of \eqref{Prob_Ori}.
	\begin{prop}
		\label{Prop_relation_DfDc_Dg}
		Suppose Assumption \ref{Assumption_f} holds. Then for any $x \in \X$ satisfying $0 \in \D_g(x) + \ca{N}_{\X}(x)$ and $c(x) = 0$, it holds that $x$ is a $(\D_f, \D_c)$-KKT point of \eqref{Prob_Ori}. 
	\end{prop}
	\begin{proof}
		For any $x \in \X$ satisfying $0 \in \D_g(x) + \ca{N}_{\X}(x)$ and $c(x) = 0$, it holds that 
		\begin{equation}
			0 \in \conv\left(\D_f(x) + \beta\D_c(x) \Tregu(c(x)) \right) + \ca{N}_{\X}(x). 
		\end{equation}
		From the convexity of $\D_f(x)$, it holds that 
		\begin{equation}
			0 \in \D_f(x) + \beta\conv\left(\D_c(x) \Tregu(c(x)) \right) + \ca{N}_{\X}(x) \subseteq \D_f(x) + \mathrm{span}(\D_c(x) ) + \ca{N}_{\X}(x). 
		\end{equation}
		This shows that $x$ is a $(\D_f, \D_c)$-KKT point of \eqref{Prob_Ori}. 
	  \end{proof}
	
	Next, we introduce the following condition on the weak Morse-Sard property of $g$. 
	\begin{assumpt}
		\label{Assumption_weak_sard}
		The set $\{g(x): x \in \X, ~0 \in \D_g(x) + \NX(x)\}$ has empty interior in $\bb{R}$. 
	\end{assumpt}
	
	\begin{rmk}
		It is worth mentioning that, from  Proposition \ref{Prop_definable_regularity} and the definability of $f$, $\D_f$, $c$ and $\D_c$, the set $\{g(x): x \in \X,~ 0 \in \D_g(x) + \ca{N}_{\X}(x)\}$ is finite and hence has empty interior in $\bb{R}$. Therefore, Assumption \ref{Assumption_weak_sard} is a mild assumption in practical scenarios, as demonstrated in \cite{davis2020stochastic,bolte2021conservative}. 
	\end{rmk}

	Here we present the assumptions on the framework \eqref{Eq_Framework}. 
	\begin{assumpt}
		\label{Assumption_framework}
		\begin{enumerate}
			\item The sequence of iterates $\{(\xk, \yk)\}$ is uniformly bounded in $\X \times \Rq$. That is, $\sup_{k\geq 0} \norm{\xk} + \norm{\yk} < +\infty$. 
			\item There exists a nonnegative sequence $\{\delta_k\}$ such that, 
			\begin{equation*}
				\lim_{k\to +\infty} \delta_k = 0, \quad d_k \in \D_f^{\delta_k}(\xk), \quad J_{c,k} \in \D_c^{\delta_k}(\xk),  \quad \norm{\wk - c(\xk)} \leq \delta_k,
			\end{equation*}
			hold  for all $k\geq 0$.
			\item The sequence $\{\xi_{k}\}$ is uniformly bounded. Moreover, the sequence of stepsizes $\{\eta_k\}$ is positive and satisfies
			\begin{equation}
				\lim_{k\to \infty} \eta_k = 0, \quad \sum_{k = 0}^{\infty} \eta_k = +\infty, \quad \lim_{i\to \infty}\sup_{i\leq j\leq \Lambda(\lambda_i+ T)} \norm{\sum_{k = i}^{j} \eta_k \xi_{k+1}} = 0,
			\end{equation}
			for any $T>0$ and any $k\geq 0$. 
			\item There exists $0 < \theta_{\min} \leq \theta_{\max} < \beta$ such that  $\theta_{\min} \leq \theta_k \leq \theta_{\max}$ holds for all $k\geq 0$. 
		\end{enumerate}
	\end{assumpt}

	In the following theorem, we present the main result of this subsection, in the sense that the sequence $\{\xk\}$ weakly converges to the set $\{ x \in \X: 0 \in \D_g(x) + \ca{N}_{\X}(x)\}$. For a clearer presentation of the main results, the proof of Theorem \ref{Theo_convergence_Dg} is given in Appendix \ref{Appendix_proof_Theo_convergence}. 
	\begin{theo}
		\label{Theo_convergence_Dg}
		For any given sequence of set-valued mappings $\{\Phi_k\}$ that represents an embeddable stochastic subgradient method, suppose Assumption \ref{Assumption_f}, Assumption \ref{Assumption_weak_sard} and Assumption \ref{Assumption_framework} hold. Then for the sequence $\{\xk\}$ generated by \eqref{Eq_Framework}, any cluster point of $\{\xk\}$ lies in the subset $\left\{x \in \X: 0 \in \D_g(x) + \ca{N}_{\X}(x)\right\}$ and the sequence $\{g(\xk)\}$ converges.
	\end{theo}

	The result in Theorem \ref{Theo_convergence_Dg} illustrates that the sequence $\{\xk\}$ finds the stationary points of $g$ within $\X$. Moreover, from Proposition \ref{Prop_relation_DfDc_Dg}, the sequence $\{\xk\}$ finds $(\D_f, \D_c)$-KKT points of \eqref{Prob_Ori} whenever all of its cluster points are feasible. Therefore, in the rest of this subsection, we derive a sufficient condition that enforces the feasibility of all the cluster points of $\{\xk\}$. 
	
	Motivated by \cite{li2023stochastic}, we consider the following regularity condition.
	\begin{cond}
		\label{Cond_regularity_condition}
		There exists a prefixed constant $\nu > 0$ such that for $x \in \X$, it holds that 
		\begin{equation}
			\label{Eq_regularity_condition}
			\nu \norm{c(x)} \leq \mathrm{dist}\left( - \D_c(x) c(x), \ca{N}_{\X}(x)  \right).
		\end{equation}
	\end{cond}

	As demonstrated in \cite{li2023stochastic}, Condition \ref{Cond_regularity_condition} holds for a wide range of applications \cite{li2021rate,xie2021complexity,lin2022complexity,li2023stochastic}, such as affine-equality constrained problems with possibly additional polyhedral or ball constraint sets \cite{li2021rate}, multi-class Neyman-Pearson classification problems \cite{lin2022complexity}. Moreover, as demonstrated in \cite{lin2022complexity,li2023stochastic}, Condition \ref{Cond_regularity_condition}  is neither strictly stronger nor strictly weaker than the uniform Slater’s condition in \cite{ma2020quadratically}, or the strong MFCQ condition adopted in \cite{boob2023stochastic}. Furthermore, it is easy to verify that the constraints analyzed in \cite{bolte2018nonconvex} satisfy Condition \ref{Cond_regularity_condition}.

	The following proposition illustrates that when $f$ is a globally Lipschitz continuous function and $\beta$ is sufficiently large, Condition \ref{Cond_regularity_condition}  guarantees that $g$ is an exact penalty function for \eqref{Prob_Ori}, in the sense that any $\D_g$-stationary point of \eqref{Prob_simple_min_X} satisfies $c(x) = 0$.  
	\begin{prop}
		\label{Prop_exactness_g}
		Suppose Assumption \ref{Assumption_f} and Condition \ref{Cond_regularity_condition} hold, $f$ is $M_f$-Lipschitz continuous over $\X$,  and $\beta >  \frac{M_f}{\nu}$. Then for any $x \in \X$ such that $0 \in \D_g(x) + \ca{N}_{\X}(x)$, it holds that $c(x) = 0$. 
	\end{prop}
	\begin{proof}
		For any $x \in \X$ and $c(x) \neq 0$, Condition \ref{Cond_regularity_condition} implies that  
		\begin{equation}
			\nu \leq \mathrm{dist}\left( 0, \D_c(x) \frac{c(x)}{\norm{c(x)}} + \ca{N}_{\X}(x)  \right).
		\end{equation}
		Together with the formulation of $\D_g$ and the Lipschitz continuity of $f$, it holds that 
		\begin{equation}
			\begin{aligned}
				&\mathrm{dist}\left( 0, \D_g(x) + \ca{N}_{\X}(x)  \right) =\mathrm{dist}\left( 0, \D_f(x) + \beta \D_c(x) \frac{c(x)}{\norm{c(x)}} + \rho \D_c(x)c(x) + \ca{N}_{\X}(x)  \right) \\
				\geq{}&  ( \beta  + \rho \norm{c(x)})\nu - M_f > 0. 
			\end{aligned}
		\end{equation}
		Therefore, whenever $x \in \X$ satisfies $0 \in \D_g(x) + \ca{N}_{\X}(x)$, it holds that $c(x) = 0$. 
	  \end{proof}

	The following corollary illustrates that, the sequence $\{\xk\}$ 
	generated by the framework \eqref{Eq_Framework} weakly converges to a $(\D_f, \D_c)$-KKT point of \eqref{Prob_Ori}. The proof of Corollary \ref{Coro_convergence_DfDc_exactness} directly follows from Theorem \ref{Theo_convergence_Dg} and Proposition \ref{Prop_exactness_g}, hence it is omitted for simplicity. 
	\begin{coro}
		\label{Coro_convergence_DfDc_exactness}
		For any given sequence of set-valued mappings $\{\Phi_k\}$ that represents an embeddable stochastic subgradient method, suppose  Assumption \ref{Assumption_f}, Assumption \ref{Assumption_weak_sard}, Assumption \ref{Assumption_framework}, and Condition \ref{Cond_regularity_condition}  hold. Moreover, we assume that $f$ is $M_f$-Lipschitz continuous over $\X$,  and $\beta > \frac{M_f}{\nu}$. Then any cluster point of $\{\xk\}$ is a $(\D_f, \D_c)$-KKT point of \eqref{Prob_Ori}, and the sequence of function values $\{f(\xk)\}$ converges. 
	\end{coro}

	\subsection{Application in expectation-constrained optimization}
	
	In this subsection, we illustrate how the framework \eqref{Eq_Framework} can be applied to solve \eqref{Prob_Ori}, where the constraint mapping $c$ takes an expectation formulation. Throughout this subsection, based on Assumption \ref{Assumption_f}, we make the following assumptions on the objective function $f$ and constraint mapping $c$ in \eqref{Prob_Ori}. 
	\begin{assumpt}
		\label{Assumption_f_expecation_constraints}
		\begin{enumerate}
			\item There exists two probability spaces $(\Omega_f, \ca{F}_f, P_f)$ and $(\Omega_c, \ca{F}_c, P_c)$ and two mappings $F: \Rn \times \Omega_f \to \bb{R}$ and $C: \Rn \times \Omega_c \to \Rp$ such that for almost every $\omega_f \in \Omega_f$ and $\omega_c \in \Omega_c$, $F(\cdot, \omega_x)$ and $C(\cdot, \omega_c)$ are path-differentiable over $\Rn$ that admit $\D_F(\cdot, \omega_f)$ and $\D_C(\cdot, \omega_c)$ as their conservative fields (Jacobians), respectively.  
			\item For any compact subset $\ca{Y}$ of $\Rn$, there exists $M_{\ca Y}>0$ such that for almost every $\omega_f \in \Omega_f$ and $\omega_c \in \Omega_c$, it holds that 
			\begin{equation}
				\sup_{x \in \Y} \max\left\{|F(x, \omega_f)|, \norm{C(x, \omega_c)}, \norm{\D_F(x, \omega_f)}, \norm{\D_C(x, \omega_c)}\right\} \leq M_{\ca Y}. 
			\end{equation}
			\item In \eqref{Prob_Ori}, we have $f(x) = \bb{E}_{\omega_f \sim P_f}[F(x, \omega_f)]$ and $c(x) = \bb{E}_{\omega_c \sim P_c}[C(x, \omega_c)]$ holds for any $x \in \X$. 
			\item $f$ is bounded from below. 
		\end{enumerate}
	\end{assumpt}

	As illustrated in Assumption \ref{Assumption_f_expecation_constraints}, both the objective function $f$ and the constraint mapping $c$ can be expressed as the expectation of $F$ and $C$ respectively, in the sense of Aumann's integral in Definition \ref{Defin_Aumann_integral}. Therefore, under Assumption \ref{Assumption_f_expecation_constraints}, only noisy evaluations of $f$ and $c$ are available. Hence \eqref{Prob_Ori} can be categorized as an optimization problem with expectation constraints studied by \cite{boob2023stochastic,li2023stochastic}.

	We first present the formulation of $\D_f$ and $\D_c$ based on  Assumption \ref{Assumption_f_expecation_constraints} in Lemma \ref{Le_formulation_DfDc}.  As the results directly follow from \cite{bolte2023subgradient}, we omit the proof.
	\begin{lem}[\cite{bolte2023subgradient}]
		\label{Le_formulation_DfDc}
		Suppose Assumption \ref{Assumption_f_expecation_constraints} hold. Then we have 
		\begin{itemize}
			\item $f$ is a path-differentiable function that admits a conservative field
			given as  
			\begin{equation}
				\D_f(x) =\conv\left(\bb{E}_{\omega_f \sim P_f}[\D_F(x, \omega_f)] \right).
			\end{equation}
			\item $c$ is a path-differentiable mapping that admits a conservative Jacobian
			given as 
			\begin{equation}
				\D_c(x) = \conv\left(\bb{E}_{\omega_c \sim P_c}[\D_C(x, \omega_c)] \right). 
			\end{equation}
		\end{itemize}
	\end{lem}

	It is worth mentioning that when $c$ takes a deterministic formulation, noiseless approximate evaluations of $\D_c$ and $c$ are available, hence we can directly choose $w_k = c(\xk)$ for all $k\geq 0$ in the framework \eqref{Eq_Framework}. However, when the constraint mapping $c$ takes an expectation formulation, we cannot directly choose $w_k$ as the noiseless evaluation of $c(\xk)$. To compute $\{\wk\}$ that satisfies Assumption \ref{Assumption_framework}(2) almost surely, we should track the values of  $\{c(\xk)\}$ from its noisy evaluations. Motivated by the update schemes in \cite{ghadimi2020single,chen2021solving,gurbuzbalaban2022stochastic}, we consider the following update scheme for $\{\wk\}$:
	\begin{equation}
		\label{Eq_EC_update_wk}
		\wkp = \wk - \tilde{\tau} \eta_k (\wk - C(\xk, \omega_{c,k}) ) + C(\xkp, \omega_{c,k}) - C(\xk, \omega_{c,k})
	\end{equation}
	where $\tilde{\tau} > 0$ is a given parameter. The update scheme \eqref{Eq_EC_update_wk} is single-timescale, in the sense that its stepsizes diminish at the same rate as the stepsizes for the primal variables $\{\xk\}$. Moreover, $C(\xkp, \omega_{c,k}) - C(\xk, \omega_{c,k})$ is the correction term, which is designed to drive $\{\wk\}$ as noiseless estimations of $\{c(\xk)\}$. In addition, it is worth mentioning that the update scheme  \eqref{Eq_EC_update_wk} only requires one additional evaluation of $C(\xkp, \omega_{c,k})$ in the $k$-th iteration, hence is computationally efficient in practice.

	Then under Assumption \ref{Assumption_f_expecation_constraints}, we introduce the following linearized Lagrangian-based method for expectation-constrained optimization (ECLM) based on the framework \eqref{Eq_Framework}, 
	\begin{equation}
		\label{Eq_Framework_EC}
		\tag{ECLM}
		\left\{
		\begin{aligned}
			& \text{Randomly and independently draw $\omega_{f,k} \sim P_f$ and $\omega_{c,k}, \tilde{\omega}_{c,k} \sim P_c$ },\\
			& g_k \in \D_F(\xk, \omega_{f,k}), \quad \hat{J}_{c,k} \in \D_C(\xk, \tilde{\omega}_{c,k}),\\
			& l_{x,k} = g_k + \hat{J}_{c,k} ( \lambda_k + \rho \wk ),\\
			& (\xkp, \ykp) \in \Phi_k(l_{x,k}, \xk, \yk, \eta_k), \\
			& \wkp = \wk - \tilde{\tau} \eta_k (\wk - C(\xk, \omega_{c,k}) ) + C(\xkp, \omega_{c,k}) - C(\xk, \omega_{c,k}), \\
			& \lambda_{k+1} = \lambda_k + \theta_k \left( \mathrm{regu}(\wkp) - \frac{1}{\beta}\lambda_k  \right).
		\end{aligned}
		\right.
	\end{equation}
	Here $\hat{J}_{c,k}$ refers to the noisy evaluation of the Jacobian of $c$ at $\xk$.

	In the following lemma, we show that the sequence $\{w_k\}$ in \eqref{Eq_Framework_EC} asymptotically serves as an approximate noiseless evaluation of $\{c(\xk)\}$. 
	\begin{lem}
		\label{Le_expectation_wk}
		Suppose Assumption \ref{Assumption_f_expecation_constraints} holds, and the framework \eqref{Eq_Framework_EC} satisfies the following conditions, 
		\begin{enumerate}
			\item The sequence of iterates $\{\xk\}$ is uniformly bounded almost surely.
			\item The sequence $\{\eta_k\}$ is positive and satisfies $\sum_{k = 0}^{+\infty} \eta_k = +\infty$ and $\lim_{k\to +\infty} |\eta_k \log(k)| = 0$.  
		\end{enumerate}
		Then almost surely, we have 
		\begin{equation}
			\lim_{k\to +\infty} \norm{\wk - c(\xk)} = 0.
		\end{equation}
	\end{lem}
	\begin{proof}
		Let $u_k := \wk - c(\xk)$ and $\zeta_{x, \omega} = C(x, \omega) - c(x)$. Then for any $k\geq 0$, it holds that 
		\begin{equation}
			\bb{E}_{\omega_{c,k} \sim P_c}[\zeta_{\xk, \omega_{c,k}}  |\ca{F}_k] = 0, \quad  \bb{E}_{\omega_{c,k} \sim P_c}[\zeta_{\xkp, \omega_{c,k}} - \zeta_{\xk, \omega_{c,k}} |\ca{F}_k] = 0.  
		\end{equation}
		Moreover, from the local Lipschitz continuity of $C(\cdot, \omega_{c,k})$, we can conclude that there exists a locally bounded function $\widetilde T_\Phi$, such that $\norm{\zeta_{\xkp, \omega_{c,k}} - \zeta_{\xk, \omega_{c,k}}} \leq \eta_k \widetilde T_{\Phi}(\xk, \yk)$ holds for any $k\geq 0$ and almost every $\omega_{c,k} \in \Omega_c$.
		
		As a result, together with the fact that $\{\xk\}$ is uniformly bounded almost surely, we can conclude that $\{\frac{1}{\eta_k}(\zeta_{\xkp, \omega_{c,k}} - \zeta_{\xk, \omega_{c,k}})\}$ is a uniformly bounded martingale difference sequence. 
		Then the update schemes for $\{u_k\}$ can be rewritten as,
		\begin{equation}
			u_{k+1}  = (1- \tilde{\tau}\eta_k)u_k + \eta_k \left( \tilde{\tau} \zeta_{\xk, {\omega}_{c,k}} +  \frac{\zeta_{\xkp, \omega_{c,k}} - \zeta_{\xk, \omega_{c,k}}}{\eta_k}\right).  
		\end{equation}
		Let $\hat{\zeta}_k =  \tilde{\tau}\zeta_{\xk, {\omega}_{c,k}} +  \frac{\zeta_{\xkp, \omega_{c,k}} - \zeta_{\xk, \omega_{c,k}}}{\eta_k}$, it is easy to verify that $\{\hat{\zeta}_k\}$ is a sequence of uniformly bounded martingale difference sequence. 
		
		Furthermore, it holds from the uniform boundedness of $\{\hat{\zeta}_k\}$ that $\{u_k\}$ is uniformly bounded almost surely. Then it follows \cite{benaim2005stochastic} that the interpolated process of $\{u_k\}$ is a perturbed solution to the ODE $\frac{\mathrm{d}u}{\mathrm{d}t} = - \tilde{\tau}u$. Therefore, we can conclude from \cite[Proposition 3.33]{benaim2005stochastic} that any cluster point of $\{u_k\}$ lies in $\{0\}$. Then almost surely, $\lim_{k\to +\infty} u_k = 0$. This completes the proof.

	  \end{proof}
	
	Lemma \ref{Le_expectation_wk} shows that with Assumption \ref{Assumption_framework}(1)(3)(4), the sequence $\{w_k\}$ in \eqref{Eq_Framework_EC} satisfies Assumption \ref{Assumption_framework}(2) almost surely. Then based on Theorem \ref{Theo_convergence_Dg}, we have the following theorem establishing the convergence properties of \eqref{Eq_Framework_EC}. The proof of theorem \eqref{Theo_convergence_Dg_expectation} directly follows from Lemma \ref{Le_expectation_wk}, Theorem \ref{Theo_convergence_Dg} and Corollary \ref{Coro_convergence_DfDc_exactness}. Hence it is omitted for simplicity. 
	\begin{theo}
		\label{Theo_convergence_Dg_expectation}
		For any given sequence of set-valued mappings $\{\Phi_k\}$ that represents an embeddable stochastic subgradient method, suppose Assumption \ref{Assumption_f_expecation_constraints} holds, and the framework \eqref{Eq_Framework_EC} satisfies the following conditions.
		\begin{enumerate}
			\item The sequence of iterates $\{\xk\}$ is uniformly bounded almost surely.
			\item The sequence $\{\eta_k\}$ is positive and satisfies $\sum_{k = 0}^{+\infty} \eta_k = +\infty$ and $\lim_{k\to +\infty} |\eta_k \log(k)| = 0$. 
			\item There exists $0 < \theta_{\min} \leq \theta_{\max} < \beta$ such that  $\theta_{\min} \leq \theta_k \leq \theta_{\max}$ holds for all $k\geq 0$. 
		\end{enumerate}
		Then for any sequence $\{\xk\}$ generated by \eqref{Eq_Framework_EC}, almost surely we have that each of its cluster points lies in $\left\{x \in \X: 0 \in \D_g(x) + \ca{N}_{\X}(x) \right\}$ and the sequence $\{g(\xk)\}$ converges.  
		
		Moreover, when we further assume the validity of Condition \ref{Cond_regularity_condition},  and $f$ is $M_f$-Lipschitz continuous and $\beta > \frac{M_f}{\nu}$, then any cluster point of $\{\xk\}$ is a $(\D_f, \D_c)$-KKT point of \eqref{Prob_Ori} and the sequence $\{f(\xk)\}$ converges. 
	\end{theo}

	\section{Applications}
	In this section, we present illustrative examples of how to embed stochastic subgradient methods into Lagrangian-based methods using the framework \eqref{Eq_Framework}. In particular, we show that the existing proximal variants of SGD \cite{davis2020stochastic}, momentum SGD \cite{ruszczynski2020convergence} are embeddable. Moreover, motivated by \cite{xiao2023adam}, we propose a proximal version of ADAM \cite{kingma2014adam} and prove that it can be embedded into the framework \eqref{Eq_Framework}. These results further reveal that a wide range of stochastic subgradient methods can be embedded into Lagrangian-based methods, with convergence guarantees directly inherited from existing works.

	\subsection{Embedding proximal SGD into Lagrangian-based method}
	In this subsection, we consider the proximal SGD method analyzed in \cite{davis2020stochastic} for solving \eqref{Prob_simple_min_X}, which iterates by the following update scheme,
	\begin{equation}
		\label{Eq_example_SGD}
		\xkp = \prox_{\X}(\xk - \eta_k g_k).
	\end{equation}
	Here $g_k \in \D_h^{\delta_k}(\xk) + \xi_{k+1}$ estimates a stochastic subgradient of $h$ at $\xk$.  For any $k\geq 0$, let the set-valued mapping $\Phi_{G,k}: \Rn \times \X \times \bb{R} \rightrightarrows \X$ be defined as, 
	\begin{equation}
		\Phi_{G,k}(g, x, \eta) = \{\prox_{\X}(x - \eta g) \},
		\label{eq-PhiG}
	\end{equation}
	then the update scheme of \eqref{Eq_example_SGD} can be expressed as 
	\begin{equation}
		\xkp \in \Phi_{G,k}(g_k, \xk, \eta_k) \subseteq  \prox_{\X}\left(\xk - \eta_k (\D_h^{\delta_k}(\xk) + \xi_{k+1} )\right). 
	\end{equation}
	Here we treat $x \in \Rn$ as primal variables, and there is no auxiliary variable in \eqref{Eq_example_SGD}.

	Following the results in \cite[Section 6]{davis2020stochastic}, we can conclude that the sequence of mappings $\{\Phi_{G,k}\}$ represents an embeddable stochastic subgradient method.  We summarize these results in the following proposition, where the proof directly follows the same proof techniques as in \cite[Section 6]{davis2020stochastic}. Therefore, we omit the proof of Proposition \ref{Prop_example_stable_SGD} for simplicity. 
	\begin{prop}[\cite{davis2020stochastic}]
		\label{Prop_example_stable_SGD}
		The sequence of set-valued mappings $\{\Phi_{G,k}\}$ in \eqref{eq-PhiG} represents an embeddable stochastic subgradient method.
	\end{prop}

	Therefore, we arrive at the following Lagrangian-based method, by embedding the proximal SGD in \eqref{Eq_example_SGD} through the framework \eqref{Eq_Framework}, 
	\begin{equation}
		\label{Eq_example_SGD_LALM}
		\tag{SGD-LALM}
		\left\{
		\begin{aligned}
			& l_{x,k} = d_k + J_{c,k} (\lambda_k + \rho \wk) + \xi_{k+1},\\
			& \text{\fcolorbox{white}{gray!20}{$\xkp = \mathrm{prox}_{\X}(\xk - \eta_k l_{x,k}) $,}}\\
            & \text{Compute $\wkp$ as an approximated evaluation to $c(\xkp)$},\\
			& \lambda_{k+1} = \lambda_k + \theta_k \left( \mathrm{regu}(\wkp) - \frac{1}{\beta}\lambda_k  \right).
		\end{aligned}
		\right.
	\end{equation}
	Here the background of the second step in \eqref{Eq_example_SGD_LALM} is shaded in grey to highlight the embedding of the proximal SGD exhibited in  \eqref{Eq_example_SGD} into the framework \eqref{Eq_Framework}.

	Based on the results from Theorem \ref{Theo_convergence_Dg} and Corollary \ref{Coro_convergence_DfDc_exactness}, we have the following theorem illustrating the global convergence of \eqref{Eq_example_SGD_LALM}. Theorem \ref{Prop_example_SGD_convergence_Dg} is a direct corollary of Theorem \ref{Theo_convergence_Dg} and Corollary \ref{Coro_convergence_DfDc_exactness}, hence its proof is omitted for simplicity. 
	\begin{theo}
		\label{Prop_example_SGD_convergence_Dg}
		Suppose Assumption \ref{Assumption_f}, Assumption \ref{Assumption_weak_sard}, and Assumption \ref{Assumption_framework} hold. Then for any sequence $\{\xk\}$ generated by \eqref{Eq_example_SGD_LALM}, any cluster point of $\{\xk\}$ lies in $\{ x \in \X: 0 \in \D_g(x) + \NX(x)\}$, and the sequence $\{g(\xk)\}$ converges. 
		
		Moreover, when we further assume Condition \ref{Cond_regularity_condition}  holds, $f$ is $M_f$-Lipschitz continuous and $\beta > \frac{M_f}{\nu}$ in \eqref{Eq_example_SGD_LALM}, then any cluster point of $\{\xk\}$ is a $(\D_f, \D_c)$-KKT point of \eqref{Prob_Ori}, and the sequence $\{f(\xk)\}$ converges. 
	\end{theo}

	
	\subsection{Embedding proximal SGDM into Lagrangian-based methods}
	
	Proximal SGD with momentum (SGDM) \cite{ruszczynski2020convergence} is a popular variant of the proximal SGD method, which introduces the heavy-ball momentum term  \cite{polyak1964some} to track the first-order moment of the stochastic subgradient of $f$. When applied to solve \eqref{Prob_simple_min_X}, \cite{ruszczynski2020convergence} introduces a proximal SGDM, which enjoys convergence guarantees on nonsmooth functions based on the ODE approaches. 
	
	In this subsection, we present an example to illustrate how to embed the proximal SGDM into Lagrangian-based methods, and establish its convergence properties directly from the results in   \cite{ruszczynski2020convergence}.  Following the settings in \cite{ruszczynski2020convergence}, we consider the proximal SGDM that solves \eqref{Prob_simple_min_X} by the following schemes,
	\begin{equation}
		\label{Eq_example_SGDM_ST}
		\left\{
		\begin{aligned}
			&\ykp = \yk - \tau \eta_k (\yk - g_k),\\
			&\xkp = (1-\eta_k) \xk +  \eta_k \mathrm{prox}_{\X}(\xk - \alpha \ykp). 
		\end{aligned}
		\right.
	\end{equation}
	Here $g_k \in \D_h^{\delta_k}(\xk) + \xi_{k+1}$ is a stochastic subgradient of $h$ at $\xk$. Moreover,  $\tau> 0$  is the parameter for the momentum term in \eqref{Eq_example_SGDM_ST}, and $\alpha>0$ is a prefixed constant.   The corresponding mapping $\Phi_{S, k}: \X\times \mathbb{R}^n\times \mathbb{R} \rightrightarrows \X \times \mathbb{R}^n$ can be expressed as 
	\begin{equation}
		\Phi_{S,k}(g, x, y,\eta) = \left\{
		\left[
		\begin{matrix}
			(1-\eta) x +  \eta \mathrm{prox}_{\X}(x - \alpha (y - \tau \eta (y-g)))\\
			y - \tau \eta (y-g) 
		\end{matrix}
		\right]\right\}.
	\end{equation}
	Here we treat $x \in \Rn$ as primal variables and $y \in \Rn$ as the auxiliary variables.
	Therefore, with the mappings $\{\Phi_{S, k}\}$, the update scheme \eqref{Eq_example_SGDM_ST} can be reformulated as 
	\begin{equation}
		\label{Eq_example_SGDM_abstract_scheme}
		(\xkp, \ykp) \in \Phi_{S,k}(\D_h^{\delta_k}(\xk) + \xi_{k+1}, \xk, \yk, \eta_k). 
	\end{equation}
	Then let the set-valued mapping $\ca{H}_S:  \X \times \Rn \rightrightarrows \Rn \times \Rn $  be defined as 
	\begin{equation}
		\ca{H}_{S}(x, y) = 
		\left[
		\begin{matrix}
			\mathrm{prox}_{\X}(x - \alpha y) - x\\[3pt]
			\tau (\D_h(x)-y)\\
		\end{matrix}
		\right],
	\end{equation}
	and consider the following differential inclusion
	\begin{equation}
		\label{Eq_example_SGDM_DI}
		\begin{aligned}
			\left( \frac{\mathrm{d}x}{\mathrm{d}t}, \frac{\mathrm{d}y}{\mathrm{d}t} \right) \in \ca{H}_{S}(x, y).			
		\end{aligned}
	\end{equation}
	In \cite[Theorem 4.2]{ruszczynski2020convergence}, the author proves that under mild conditions, the differential inclusion \eqref{Eq_example_SGDM_DI} admits the Lyapunov function $\Psi_{S}(x, y) = h(x) - \frac{1}{\tau}u_{S}(x, y)$ with the stable set $\{(x, y) \in \Rn \times \Rn: 0 \in \D_h(x) + \ca{N}_{\X}(x), u_{S}(x, y) = 0\}$. Here $u_S: \X \times \Rn \to \bb{R}$ is an auxiliary function defined by 
	\begin{equation}
		u_{S}(x, y) := \min_{w \in \X} \inner{w-x, y} + \frac{\alpha}{2} \norm{w-x}^2. 
	\end{equation}
	In the following proposition, we summarize these results from \cite{ruszczynski2020convergence}.
	\begin{prop}[Theorem 4.2 in \cite{ruszczynski2020convergence}]
		\label{Prop_example_SGDM_perturbed_path}
		For any locally Lipschitz path-differentiable function $h$ that admits $\D_h$ as its convex-valued conservative field, the function $\Psi_S(x, y)$ is a Lyapunov function for the differential inclusion \eqref{Eq_example_SGDM_DI} 
		with the stable set $\{(x, y) \in \X \times \Rn: 
		0 \in \D_h(x) + \ca{N}_{\X}(x), u_{S}(x, y) = 0\}$. 
	\end{prop}

	Then based on existing ODE approaches \cite{benaim2005stochastic,davis2020stochastic,xiao2023convergence}, we have the following proposition showing that \eqref{Eq_example_SGDM_ST} is an embeddable algorithm represented by the sequence of set-valued mappings $\{\Phi_{S, k}\}$. The proof of the following proposition follows almost the same techniques as \cite[Theorem 4.2]{ruszczynski2020convergence}. However, it is worth mentioning that the results in \cite{ruszczynski2020convergence} impose stronger conditions on the noises in \eqref{Eq_example_SGDM_ST}, hence we include the proof of Proposition \ref{Prop_example_SGDM_stable_method} in the following. 
	\begin{prop}
		\label{Prop_example_SGDM_stable_method}
		The sequence of mappings $\{\Phi_{S,k}\}$ represents an embeddable stochastic subgradient method. 
	\end{prop}
	\begin{proof}
		Under Condition \ref{Cond_stable_alg}, since $\{\eta_k\}$ is diminishing, we can conclude that $\lim_{k\to +\infty} \norm{\xkp - \xk} + \norm{\ykp - \yk} = 0$. Therefore, from the graph-closedness of $\ca{H}_{S}$, we can conclude that there exists a diminishing sequence $\{\tilde{\delta}_k\}$ such that 
		\begin{equation}
			(\xkp, \ykp) \in \bigcup_{d_x \in \D_h^{\delta_k}(\xk) + \xi_{k+1}}\Phi_{S, k}(d_x, \xk, \yk, \eta_k) \subseteq (\xk, \yk) + \eta_k \left(\ca{H}_{S}^{\tilde{\delta}_k}(\xk, \yk) + (0, \tau\xi_{k+1}) \right).
		\end{equation}
		Note that in deriving the above inclusion, we have used the 
		non-expansive property of $\mathrm{prox}_{\X}(\cdot).$
		
		Next, from Condition \ref{Cond_stable_alg}(4),  Lemma \ref{Le_interpolated_process} illustrates that the interpolated process of $\{(\xk, \yk)\}$ is a perturbed solution to the differential inclusion \eqref{Eq_example_SGDM_DI}. Finally, from Theorem \ref{The_convergence_beniam}, Proposition \ref{Prop_example_SGDM_perturbed_path}, and the fact that $\{h(x): 0 \in \D_h(x) + \NX(x), ~x \in \X\}$ has empty interior, we have that any cluster point of $\{(\xk, \yk)\}$ lies in the set $\{(x, y) \in \Rn \times \Rn: 0 \in \D_h(x) + \ca{N}_{\X}(x), u_{S}(x, y) = 0\}$, and the sequence $\{\Psi_{S}(\xk, \yk)\}$ converges. Noticing that  $\lim_{k\to +\infty} u_{S}(\xk, \yk) = 0$, we can conclude that the sequence $\{h(\xk)\}$ converges. 
		
		As a result, from Definition \ref{Defin_stable_stochastic_method}, the proximal SGDM method in \eqref{Eq_example_SGDM_ST} is an embeddable stochastic subgradient method represented by $\{\Phi_{S, k}\}$. This completes the proof. 
	  \end{proof}

	Therefore, by embedding the proximal SGDM described in \eqref{Eq_example_SGDM_ST} through the framework \eqref{Eq_Framework}, we obtain the following Lagrangian-based method named (SGDM-LALM), 
	\begin{equation}
		\label{Eq_example_SGDM_LALM}
		\tag{SGDM-LALM}
		\left\{
		\begin{aligned}
			& l_{x,k} = d_k + J_{c,k} (\lambda_k + \rho \wk) + \xi_{k+1},\\
			& \text{\fcolorbox{white}{gray!20}{$\ykp = \yk - \tau \eta_k (\yk - l_{x,k}) $,}}\\
			& \text{\fcolorbox{white}{gray!20}{$\xkp = (1-\eta_k) \xk +  \eta_k \mathrm{prox}_{\X}(\xk - \alpha \ykp) $,}}\\
            & \text{Compute $\wkp$ as an approximated evaluation to $c(\xkp)$},\\
			& \lambda_{k+1} = \lambda_k + \theta_k \left( \mathrm{regu}(\wkp) - \frac{1}{\beta}\lambda_k  \right).
		\end{aligned}
		\right.
	\end{equation}
	Here the second step and the third step of \eqref{Eq_example_SGD_LALM} are shaded in grey to highlight the embedding of the proximal SGDM exhibited in  \eqref{Eq_example_SGDM_ST} into the framework \eqref{Eq_Framework}. 
	
	Finally,  we present Theorem \ref{Prop_example_SGDM_convergence_Dg} to illustrate the convergence properties of \eqref{Eq_example_SGDM_LALM}. The proof of Theorem \ref{Prop_example_SGDM_convergence_Dg} directly follows from Theorem \ref{Theo_convergence_Dg} and Corollary \ref{Coro_convergence_DfDc_exactness}, hence is omitted for simplicity. 
	\begin{theo}
		\label{Prop_example_SGDM_convergence_Dg}
		Suppose Assumption \ref{Assumption_f}, Assumption \ref{Assumption_weak_sard}, and Assumption \ref{Assumption_framework} hold. Then for any sequence $\{\xk\}$ generated by \eqref{Eq_example_SGDM_LALM}, any cluster point of $\{\xk\}$ lies in $\{ x \in \X: 0 \in \D_g(x) + \NX(x)\}$, and the sequence $\{g(\xk)\}$ converges. 
		
		Moreover, when we further assume that Condition \ref{Cond_regularity_condition}  holds,
		and  $f$ is $M_f$-Lipschitz continuous and $\beta > \frac{M_f}{\nu}$ in \eqref{Eq_example_SGDM_LALM}, then any cluster point of $\{\xk\}$ is a $(\D_f, \D_c)$-KKT point of \eqref{Prob_Ori}, and the sequence $\{f(\xk)\}$ converges. 
	\end{theo}

	\subsection{Embedding proximal ADAM into Lagrangian-based methods}
	In this subsection, we first investigate how to develop a proximal ADAM that enjoys convergence guarantees for solving \eqref{Prob_simple_min_X} based on the ADAM for unconstrained optimization in \cite{xiao2023adam}. Then we show that our proposed framework \eqref{Eq_Framework} enables us to embed the proximal ADAM into Lagrangian-based methods directly.

	Throughout this subsection, we denote the preconditioned proximal mapping as 
	\begin{equation}
		\pprox_{\X}(x, y; v):= \mathop{\arg\min}_{z \in \X} \left\{\inner{y,z-x}+\frac{1}{2}\inner{v \odot (z-x), z-x}\right\}. 
	\end{equation}
	Based on the notation of $\pprox$ and the update scheme in \cite[Equation (AFM)]{xiao2023adam},  we consider the following proximal ADAM for solving \eqref{Prob_simple_min_X},  
	\begin{equation}
		\label{Eq_example_ADAM_ST}
		\left\{
		\begin{aligned}
			&\ykp = \yk - \tau_1 \eta_k (\yk - g_k),\\
			& \vkp = \vk - \tau_{2} \eta_k (\vk - g_k \odot g_k),\\
			&\xkp = (1-\eta_k) \xk + \eta_k \pprox_{\X}\left(\xk , \ykp; \frac{1}{\alpha}\sqrt{\vkp + \varepsilon}\right). 
		\end{aligned}
		\right.
	\end{equation}
	Here $g_k \in \D_h^{\delta_k}(\xk) + \xi_{k+1}$ estimates a stochastic subgradient of $h$ at $\xk$, while $\tau_1, \tau_2$ satisfying  $4\tau_1 \geq \tau_2 > 0$, are parameters for the momentum terms $\{\yk\}$ and variance estimators $\{\vk\}$, respectively. Moreover, $\alpha> 0$ is a prefixed constant. It is easy to verify that when $\X = \Rn$, the proximal ADAM in \eqref{Eq_example_ADAM_ST} reduces to the ADAM analyzed in \cite{xiao2023adam}. 
	
	Then the mapping  $\Phi_{A,k}: \X\times \mathbb{R}^n\times \mathbb{R}^n \times \mathbb{R} \rightrightarrows \X \times \mathbb{R}^n\times \mathbb{R}^n$ corresponding to the proximal ADAM in \eqref{Eq_example_ADAM_ST} can be expressed as follows, 
	\begin{equation*}
		\begin{aligned}
			\Phi_{A,k}(g,x,y,v, \eta) := 
			\left\{\left[
			\begin{matrix}
				(1-\eta) x + \eta \pprox_{\X}\left(x, y - \tau_1 \eta(y-g ); \frac{1}{\alpha}\sqrt{v - \tau_2 \eta(v - g\odot g) + \varepsilon}\right)\\
				y - \tau_1 \eta (y-g)\\
				v - \tau_2 \eta(v - g\odot g)
			\end{matrix}
			\right]\right\}.
		\end{aligned}
	\end{equation*}
	Therefore, the proximal ADAM in \eqref{Eq_example_ADAM_ST} can be reformulated as 
	\begin{equation}
		(\xkp, \ykp, \vkp) \in \Phi_{A,k}(g_k, \xk, \yk, \vk, \eta_k).
	\end{equation}
	Here we treat $x \in \Rn$ as primal variables and $(y,v) \in \bb{R}^{2n}$ as the auxiliary variables.

	To verify the embeddability of the proximal ADAM in \eqref{Eq_example_ADAM_ST}, we first define the auxiliary function $u_{A}: \X \times \Rn \times \Rn_+ \to \bb{R}$ and the mapping   $\tilde{z}_A: \X \times \Rn \times \Rn_+ \to \X$ as follows, 
	\begin{align}
		u_A(x, y, v) :={}& \min_{z \in \X} ~\inner{z-x, y} + \frac{1}{2\alpha}  \inner{  (v + \varepsilon)^{\frac{1}{2}} \odot (z-x), z-x},\\
		\tilde{z}_A(x,y,v) :={}& \mathop{\arg\min}_{z \in \X} ~\inner{z-x, y} + \frac{1}{2\alpha}  \inner{  (v + \varepsilon)^{\frac{1}{2}} \odot (z-x), z-x}.
	\end{align}
	Then the following lemma gives a closed-form expression of the gradients of $u_A$, which follows from direct calculations. We omit its proof for simplicity. 
	\begin{lem}
		\label{Le_example_ADAM_uA}
		For any $(x,y,v) \in \X \times \Rn \times \Rn_+$, it holds that 
		\begin{equation}
			\begin{aligned}
				&\nabla_{x} u_A(x,y,v) = -y + \frac{1}{\alpha} (v + \varepsilon)^{\frac{1}{2}} \odot (x - \tilde{z}_A(x,y,v)),\\
				&\nabla_{y} u_A(x,y,v) = \tilde{z}_A(x,y,v) - x,\\
				&\nabla_{v} u_A(x,y,v) = \frac{1}{4\alpha} (v+ \varepsilon)^{-\frac{1}{2}} \odot (\tilde{z}_A(x,y,v)-x)^2. 
			\end{aligned}
		\end{equation}
	\end{lem}

	Moreover, let $\Psi_{A}: \X \times \Rn \times \Rn_+ \to \bb{R}$ be defined as 
	\begin{equation}
		\Psi_{A}(x,y,v):= h(x) - \frac{1}{\tau_1} u_A(x,y,v), 
	\end{equation}
	and we define the  mapping $\ca{H}_A: \X \times \Rn \times \Rn_+ \rightrightarrows \Rn \times \Rn \times \Rn_+$ as, 
	\begin{equation}
		\ca{H}_{A}(x,y,v) = 
		\left[
		\begin{matrix}
			-x + \pprox_{\X} (x, y; \frac{1}{\alpha}\sqrt{v + \varepsilon})\\
			\tau_1(\D_h(x) - y)\\
			\tau_2(\conv\left( \D_h(x) \odot \D_h(x)  \right) - v)\\
		\end{matrix}
		\right].
	\end{equation}
	Then the following proposition illustrates that $\Psi$ is a Lyapunov function for the differential inclusion 
	\begin{equation}
		\label{Eq_example_Adam_DI}
		\left( \frac{\mathrm{d}x}{\mathrm{d}t}, \frac{\mathrm{d}y}{\mathrm{d}t}, \frac{\mathrm{d}v}{\mathrm{d}t}  \right) \in \ca{H}_A(x,y,v).
	\end{equation}
	\begin{prop}
		\label{Prop_example_ADAM_perturbed_path}
		Suppose $h$ is path-differentiable and admits $\D_h$ as its conservative field, and $0< \tau_2\leq 4\tau_1$. Then the function $\Psi$ is a Lyapunov function for the differential inclusion \eqref{Eq_example_Adam_DI}, with the stable set $\{(x, y, v) \in \X \times \Rn \times \Rn_+: 0 \in \D_h(x) + \NX(x), u_A(x,y,v) = 0 \}$. 
	\end{prop}
	\begin{proof}
		For any trajectory of the differential inclusion \eqref{Eq_example_Adam_DI} with $x(s) \in \X$ for almost every $s \in \bb{R}_+$, there exists $l_h: \bb{R}_+ \to \Rn$ and  $\tilde{l}_h: \bb{R}_+ \to \Rn$ such that $l_h(s), \tilde{l}_h(s) \in \D_h(x(t))$, $\dot{x}(s) = -x(s) + \pprox_{\X} (x(s), y(s); \frac{1}{\alpha}\sqrt{v(s) + \varepsilon})$, $\dot{y}(s) = \tau_1(l_h(s) - y(s))$, and $\dot{v}(s) = \tau_2 (\tilde{l}_h(s) \odot \tilde{l}_h(s) - v(s))$ holds for almost every $s \in \bb{R}_+$.
		
		Let $w(s) := \pprox_{\X}(x(s),y(s); \frac{1}{\alpha}\sqrt{v(s) + \varepsilon } )$. Then from the definition of $w(s)$, 
		$q(s):= y(s) + \frac{1}{\alpha} \sqrt{v(s) + \varepsilon} \odot (w(s) - x(s)) \in -\ca{N}_{\X}(w(s))$ holds for almost every $s \geq 0$. Then it holds that 
		\begin{equation}
			\label{Eq_Prop_example_ADAM_perturbed_path_0}
			\begin{aligned}
				&\inner{y(s), w(s) - x(s)} + \frac{1}{\alpha}\inner{\sqrt{v(s) + \varepsilon} \odot (w(s) - x(s)), w(s) - x(s)}\\
				={}& \inner{y(s) + \frac{1}{\alpha} \sqrt{v(s) + \varepsilon} \odot (w(s) - x(s)), w(s) - x(s)} = \inner{-q(s)}{x(s) - w(s)}
				\leq 0. 
			\end{aligned}
		\end{equation}
		Here the last inequality follows from the fact that $-q(s)\in \ca{N}_{\X}(w(s))$ and $x(s) \in \X.$

		As a result, for almost every $s\geq 0$, we have 
		\begin{equation*}
			\begin{aligned}
				&\inner{\D_{\Psi_A}(x(s),y(s),v(s)), (\dot{x}(s), \dot{y}(s), \dot{v}(s)) } \\
				\ni{}& \inner{w(s) - x(s), l_h(s) + \frac{1}{\tau_1}y(s) + \frac{1}{\alpha \tau_1} \sqrt{v(s) + \varepsilon} \odot (x(s) - w(s))  } \\
				&+ \frac{1}{\tau_1}\inner{ \tau_1(l_h(s) - y(s)), -(w(s) - x(s))} + \frac{\tau_2}{4\alpha \tau_1}\inner{\tilde{l}_h(s) \odot \tilde{l}_h(s) - v(s), - (v(s)+\varepsilon)^{-\frac{1}{2}} \odot (w(s) - x(s))^2}\\
				\leq {}& \inner{w(s) - x(s), l_h(s)} + \inner{-(w(s) - x(s)), l_h(s) - y(s)}\\
				& + \frac{\tau_2}{4\alpha \tau_1}\inner{\tilde{l}_h(s) \odot \tilde{l}_h(s) - v(s), - (v(s)+\varepsilon)^{-\frac{1}{2}} \odot (w(s) - x(s))^2}\\
				={}& \inner{y(s), w(s) - x(s) } +  \frac{\tau_2}{4\alpha \tau_1}\inner{\tilde{l}_h(s) \odot \tilde{l}_h(s) - v(s), - (v(s)+\varepsilon)^{-\frac{1}{2}} \odot (w(s) - x(s))^2}\\
				\leq{}& \inner{y(s), w(s) - x(s) } +  \frac{\tau_2}{4\alpha \tau_1}\inner{ v(s), (v(s)+\varepsilon)^{-\frac{1}{2}} \odot (w(s) - x(s))^2}\\
				\leq{}
				& -\frac{1}{\alpha} \inner{\sqrt{v(s) + \varepsilon} \odot (w(s) - x(s)), w(s) - x(s)}+  \frac{\tau_2}{4\alpha \tau_1}\inner{ v(s), (v(s)+\varepsilon)^{-\frac{1}{2}}\odot (w(s) - x(s))^2}\\
				\leq{}& -\frac{\tau_2}{4\alpha \tau_1} \inner{\sqrt{v(s) + \varepsilon} - v(s) \odot (v(s) + \varepsilon)^{-\frac{1}{2}}, (w(s) - x(s))^2}\\
				={}& -\frac{\tau_2 \varepsilon}{4\alpha\tau_1}  \inner{(v(s) + \varepsilon)^{-\frac{1}{2}},  (w(s) - x(s))^2  }. 
			\end{aligned}
		\end{equation*}
		Here, the first inequality used \eqref{Eq_Prop_example_ADAM_perturbed_path_0}, and the second inequality uses the fact that 
		\begin{equation*}
			\inner{\tilde{l}_h(s) \odot \tilde{l}_h(s) \odot (v(s) + \varepsilon)^{-\frac{1}{2}},  (w(s) - x(s))^2} \geq 0, \quad \forall s \geq 0. 
		\end{equation*}
		Moreover, the third inequality uses \eqref{Eq_Prop_example_ADAM_perturbed_path_0} again, and the last inequality directly uses the fact that $4\tau_1 \geq \tau_2 > 0$. 
		
		As a result, we can conclude that 
		\begin{equation}
			\begin{aligned}
				&\Psi_A(x(t), y(t), v(t) ) - \Psi_A(x(0), y(0), v(0))
				\leq \int_{0}^t -\frac{\varepsilon \tau_2}{4\alpha\tau_1}\inner{(v(s) + \varepsilon)^{-\frac{1}{2}},  (w(s) - x(s))^2  } \mathrm{d}s \leq 0.  
			\end{aligned}
		\end{equation}
		
		Moreover, when   $\Psi_A(x(t), y(t), v(t) ) = \Psi_A(x(0), y(0), v(0))$ holds for all  $t>0$, we can conclude that $w(s) - x(s) = 0$ holds for almost every $s \geq 0$. Therefore, from the fact that $\dot{x}(s) \in -x(s) + w(s)$ holds for almost every $s \geq 0$, we can conclude that 
		\begin{equation}
			\label{Eq_Prop_example_ADAM_perturbed_path_1}
			x(s) = x(0) \qquad \text{for almost every $s \geq 0$.}
		\end{equation}
		Then from the definition of $\pprox$, we can conclude that $0 \in y(s) + \NX(x(s))$ holds for almost every $s \geq 0$. Together with the graph-closedness of $\NX(\cdot)$ and \eqref{Eq_Prop_example_ADAM_perturbed_path_1} it holds that 
		\begin{equation}
			\label{Eq_Prop_example_ADAM_perturbed_path_2}
			0 \in y(s) + \NX(x(0)) \qquad \text{for almost every $s \geq 0$.}
		\end{equation}
		
		Furthermore, the update scheme of $\{\yk\}$ in \eqref{Eq_example_Adam_DI} illustrates that $\dot{y}(s) \in \tau_1(\D_h(x(s)) -y(s) ) = \tau_1(\D_h(x(0)) -y(s) ) $ holds for almost every $s > 0$. Therefore,   $ \lim_{s \to +\infty}\mathrm{dist}\left(y(s), \D_h(x(0))  \right) = 0$. Then from the convexity of $\D_h(x(0))$ and \eqref{Eq_Prop_example_ADAM_perturbed_path_2},  we can conclude that $0 \in \D_h(x(0)) + \NX(x(0))$. 
		Therefore, $\Psi$ is a Lyapunov function for the differential inclusion \eqref{Eq_example_Adam_DI} and admits $\{(x, y, v) \in \X \times \Rn \times \Rn_+: 0 \in \D_h(x) + \NX(x), u_A(x,y,v) = 0 \}$ as the stable set. 
		This completes the proof. 
	  \end{proof}

	Based on Lemma \ref{Le_example_ADAM_uA}, Proposition \ref{Prop_example_ADAM_perturbed_path}, and the ODE approach introduced in \cite{benaim2005stochastic}, we obtain Proposition \ref{Prop_example_ADAM_stable_method}  showing that $\{\Phi_{A,k}\}$ represents an embeddable stochastic subgradient method. 
	\begin{prop}
		\label{Prop_example_ADAM_stable_method}
		With $0< \tau_2 \leq 4\tau_1$,  the sequence of mappings $\{\Phi_{A,k}\}$ represents an embeddable stochastic subgradient method. 
	\end{prop}
	\begin{proof}
		
		For any path-differentiable locally Lipschitz continuous function $h$ with conservative field $\D_h$,  under Condition \ref{Cond_stable_alg}, we get from the local boundedness of $\D_h$ that  $\lim_{k\to +\infty} \norm{\xkp - \xk} =0$. As a result, there exists a nonnegative diminishing sequence $\{\tilde{\delta}_k\}$ such that 
		\begin{equation}
				\begin{aligned}
					&(\xkp, \ykp, \vkp) \in  \bigcup_{d_x \in \D_h^{\delta_k}(\xk) + \xi_{k+1}} \Phi_{A,k}(d_x, \xk, \yk, \vk, \eta_k ) \\
					\subseteq{}& (\xk, \yk, \zk)  + \eta_k \left( \ca{H}^{\tilde{\delta}_k}(\xk, \yk, \zk) + (0,\tau_1 \xi_{k+1}, 0) \right). 
				\end{aligned}
		\end{equation}
		Then from Condition \ref{Cond_stable_alg},   Lemma \ref{Le_interpolated_process} illustrates that the interpolated process of $\{\xk, \yk, \vk\}$ is a perturbed solution to the differential inclusion \eqref{Eq_example_SGDM_DI}. Finally, from Theorem \ref{The_convergence_beniam}, Proposition \ref{Prop_example_ADAM_perturbed_path}, and the fact that $\{h(x): 0 \in \D_h(x) + \NX(x), ~x \in \X\}$ has empty interior, we have that any cluster point of $\{(\xk, \yk, \vk)\}$ lies in the set $\{(x, y, v) \in \Rn \times \Rn \times \Rn_+: 0 \in \D_h(x) + \ca{N}_{\X}(x), u_{A}(x, y, v) = 0\}$, and the sequence $\{\Psi_{A}(\xk, \yk, \vk)\}$ converges. Noticing that  $\lim_{k\to +\infty} u_{A}(\xk, \yk, \vk) = 0$, we can conclude that the sequence $\{h(\xk)\}$ converges. 
		
		As a result, from Definition \ref{Defin_stable_stochastic_method}, the ADAM method in \eqref{Eq_example_ADAM_ST} is an embeddable stochastic subgradient method represented by $\{\Phi_{A, k}\}$. This completes the proof. 
	  \end{proof}

	Then, based on the framework \eqref{Eq_Framework}, we consider the following ADAM-embedded Lagrangian-based method (ADAM-LALM), 
	\begin{equation}
		\label{Eq_example_ADAM_LALM}
		\tag{ADAM-LALM}
		\left\{
		\begin{aligned}
			& l_{x,k} = d_k + J_{c,k} (\lambda_k + \rho  \wk) + \xi_{k+1},\\
			& \text{\fcolorbox{white}{gray!20}{$\ykp = \yk - \tau_{1} \eta_k (\yk - l_{x,k}) $,}}\\
			& \text{\fcolorbox{white}{gray!20}{$\vkp = \vk - \tau_{2} \eta_k (\vk - l_{x,k} \odot l_{x,k}) $,}}\\
			& \text{\fcolorbox{white}{gray!20}{$\xkp = (1-\eta_k) \xk + \eta_k \pprox_{\X}\left(\xk, \ykp; \frac{1}{\alpha}\sqrt{\vkp + \varepsilon}\right) $,}}\\
            & \text{Compute $\wkp$ as an approximated evaluation to $c(\xkp)$},\\
			& \lambda_{k+1} = \lambda_k + \theta_k \left( \mathrm{regu}(\wkp) - \frac{1}{\beta}\lambda_k  \right).
		\end{aligned}
		\right.
	\end{equation}
	Here, the second to fourth steps of \eqref{Eq_example_SGD_LALM} are shaded in grey to highlight the embedding of the proximal ADAM method, as demonstrated in \eqref{Eq_example_Adam_DI}, into the framework outlined in \eqref{Eq_Framework}.
	
	Finally, based on the convergence properties of framework \eqref{Eq_Framework}, we present the following theorem establishing the convergence properties of \eqref{Eq_example_ADAM_LALM}. 
	\begin{theo}
		Suppose Assumption \ref{Assumption_f}, Assumption \ref{Assumption_weak_sard}, and Assumption \ref{Assumption_framework} hold, and $0< \tau_2 \leq 4\tau_1$. Then, for any sequence $\{\xk\}$ generated by \eqref{Eq_example_ADAM_LALM}, any cluster point of $\{\xk\}$ lies in $\{ x \in \X: 0 \in \D_g(x) + \NX(x)\}$, and the sequence $\{g(\xk)\}$ converges. 
		
		Moreover, when we further assume that Condition \ref{Cond_regularity_condition}  holds,
		and $f$ is $M_f$-Lipschitz continuous and $\beta > \frac{M_f}{\nu}$ in \eqref{Eq_example_ADAM_LALM}, then any cluster point of $\{\xk\}$ is a $(\D_f, \D_c)$-KKT point of \eqref{Prob_Ori}, and the sequence $\{f(\xk)\}$ converges. 
	\end{theo}

	
 \section{Numerical Experiments}
	In this section, we conduct preliminary numerical experiments to demonstrate the numerical performance of our proposed framework that embeds efficient proximal stochastic subgradient methods. 
	All the numerical experiments were conducted using an NVIDIA RTX 3090 GPU and were implemented in Python 3.9 with PyTorch 1.12.0.

	Following the settings in \cite{barlaud2021learning,grigas2019stochastic,pokutta2020deep}, we focus on training an $L$-layer nonsmooth neural network with $L_1$ ball and box constraints. The detailed optimization problem can be expressed as 
	\begin{equation}
		\min_{x\in\mathbb{R}^n}\;\{f(x):\;\|x_i\|_1\leq r_i,\;x\in\mathcal{K}\subset\mathbb{R}^n,\;i=1,...,L\}.
		\label{Eq_Experiment_Prob_L1}
	\end{equation}
	Here $f$ represents the loss function of the neural network, which is chosen as  LeNet10 \cite{lecun1998gradient} on MNIST \cite{lecun1998mnist} and ResNet14 \cite{he2016deep} on CIFAR-10 \cite{krizhevsky2009learning} datasets in our numerical experiments. 
	By introducing a slack variable $s\in\mathbb{R}^L$ and $\tilde x:=(x,s)\in\mathbb{R}^{n+L}$, we can reformulate problem \eqref{Eq_Experiment_Prob_L1} as the following problem in the form of \eqref{Prob_Ori}:
	\begin{equation}
		\label{Eq_Experiment_Prob_L1ball}
		\begin{aligned}
			\min_{\tilde x=(x,s)\in\mathcal{K}\times\mathbb{R}^L_+}\; \big\{f(x)
			: \;c(\tilde x)=0\big\},
		\end{aligned}
	\end{equation}
	where $c(\tilde x)=(c_1(x_1),...,c_L(x_L))$ and $c_i(\tilde x_i)=\|x_i\|_1+s_i-r_i$, for $i=1,...,L$.
	We set $r_i=1$ for all $i=1,...,L$ and $\mathcal{K}=[-1,1]^n$ in our experiments. It is worth mentioning that we employ the ReLU as the activation function for all the convolutional layers and full-connected layers, hence the loss function $f$ is nonsmooth but definable, as illustrated in \cite{davis2020stochastic,bolte2021conservative}.

	In our numerical experiments, we choose to embed proximal SGDM and proximal ADAM into our framework \eqref{Eq_Framework}, which results in the update schemes in \eqref{Eq_example_SGDM_LALM} and \eqref{Eq_example_ADAM_LALM}, respectively. To have better illustrations, we compare the numerical performance of these two methods with those of the stochastic inexact augmented Lagrangian method (Stoc-iALM) proposed by \cite[Algorithm 1]{li2023stochastic}. As illustrated in \cite{li2023stochastic}, the Stoc-iALM employs the P-Storm to solve its primal subproblems. However, developed based on SVRG, P-Storm usually delivers inferior performance than proximal SGDM and proximal ADAM in neural network training tasks. Therefore, to better compare these methods, we also incorporate the same proximal stochastic subgradient methods (e.g., proximal SGDM and proximal ADAM) to solve the subproblems within Stoc-iALM. These two variants of Stc-iALM are denoted as SGDM-iALM and ADAM-iALM, respectively. 
	Furthermore, notice that the convergence analysis in \cite{{li2023stochastic}} necessarities the weak convexity of the objective function $f$. Consequently, their results cannot guarantee the convergence of their proposed methods for solving \eqref{Eq_Experiment_Prob_L1}, as the loss function $f$ is typically not Clarke regular. As highlighted in Section 1, there is no prior research investigating the development of a linearized Lagrangian-based method for solving \eqref{Prob_Ori} with such $f(x)$ and $c(x)$.

	In all the test instances of our numerical experiments, we train the neural network for $400$ epochs with a batch size of $128$. Due to the non-regularity of $f$ in \eqref{Eq_Experiment_Prob_L1ball}, the suggested stopping criteria for the primal subproblems in Stoc-iALM are no longer valid. Therefore, we choose to execute the inner solvers for $500$ steps to solve the primal subproblems within SGDM-iALM and ADAM-iALM. All the other hyper-parameters in the compared methods are tuned by the following strategies.
	\begin{itemize}
		\item SGDM-LALM/SGDM-iALM: We set $\theta_k=\theta_0$ and $\eta_k=\frac{0.1}{\sqrt{s+1}}$ if the iteration number $k$ lies in the $s$-th epoch. Additionally, $\tau=1$, $\beta=1$, and we search for parameters $(\alpha, \rho, \theta_0)$ within the set $\{l\times10^{-k}:\,l=1,5,\,k=1,2,3\}\times\{10^{-k}:\,k=1,2,3,4\}\times\{10^{-k}:\,k=1,2,4,6,8\}$. We train the neural network for 20 epochs and select hyperparameters based on the best test accuracy, while ensuring the constraint feasibility is reduced by more than 50\%  compared to the initial point. In addition, as  SGDM-iALM updates its dual variables $\{\lambda_k\}$ by \eqref{Eq_Intro_dual_ascent_iALM}, we choose $\theta_k$ in \eqref{Eq_Intro_dual_ascent_iALM} as $1$ throughout the iterations.  
		\item ADAM-LALM/ADAM-iALM: We set $\tau_1=1$, $\tau_2=0.1$, $\eta_k=\frac{0.1}{\sqrt{s+1}}$, and $\epsilon=10^{-8}$ as the default setting for ADAM in PyTorch. The remaining hyperparameters follow the same strategy as in SGDM-LALM/SGDM-iALM. 
	\end{itemize}

	Figure \ref{fig:lenet-mnist} and Figure \ref{fig:lenet-cifar10} demonstrate the overall comparability of our proposed methods SGDM/ADAM-LALM with the existing augmented Lagrangian methods SGDM/ADAM-iALM proposed by \cite{li2023stochastic}. It is worth mentioning that the theoretical analysis in \cite{li2023stochastic} requires the weak convexity of $f$ and the differentiability of $c$ in \eqref{Prob_Ori}, hence lacking convergence guarantees in solving \eqref{Prob_Ori}. From Figure \ref{fig:lenet-mnist}, we can observe that SGDM-iALM and ADAM-iALM exhibit slightly better test performance, while our proposed method can achieve lower constraint feasibility. In Figure \ref{fig:lenet-cifar10}, SGD-type methods SGDM-LALM and SGDM-iALM demonstrate superior generalization performance compared to ADAM-type methods, namely ADAM-LALM and ADAM-iALM. Regarding constraint feasibility, we can observe that the performance of our proposed methods SGDM-LALM and ADAM-LALM are comparable with their counterparts.

	From these results of our numerical experiments, we have observed a trade-off between test accuracy and feasibility at the end phase of the training, as illustrated in both Figure \ref{fig:lenet-mnist} and Figure \ref{fig:lenet-cifar10}. Intuitively, a milder requirement on the constraint violation yields lower function values, leading to higher test accuracy.   
	To better illustrate this phenomenon, we employ the same SGDM-LALM/ADAM-LALM with varying parameters $\alpha$ and different $\rho$ to train LeNet18 on the MNIST dataset. The results, as shown in Figure \ref{fig:lenet-mnist-sgdm-trade-off} and Figure \ref{fig:lenet-mnist-adam-trade-off}, consistently validate the presence of this trade-off between test accuracy and constraint feasibility. For example, in Figure \ref{fig:lenet-mnist-adam-trade-off}, ADAM-LALM with $\rho=10^{-2}$ and $\rho=10^{-5}$ exhibit a remarkable reduction in feasibility compared to the other two configurations.  However, this reduction comes at the cost of a slower training loss decrease and a lower test accuracy.

	\begin{figure*}[tb]
		\caption{Numerical results on training LeNet10 on MNIST with constraints.} \label{fig:lenet-mnist}
		\begin{center}
			\setlength{\tabcolsep}{0.0pt}  
			\scalebox{1}{\begin{tabular}{ccc}
					\includegraphics[width=0.33\linewidth]{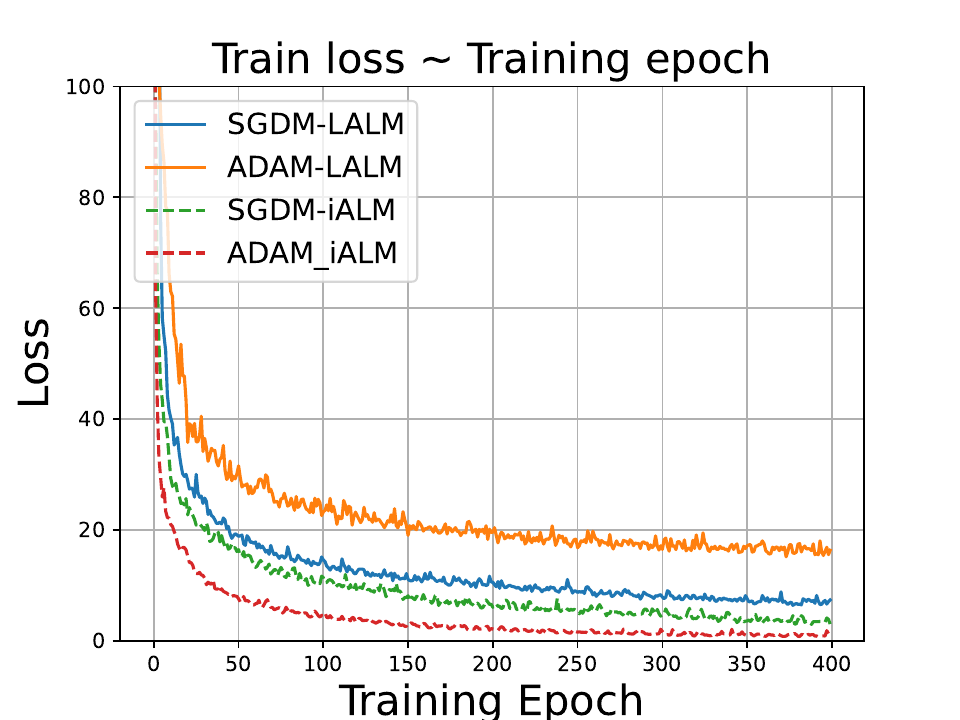}&
					\includegraphics[width=0.33\linewidth]{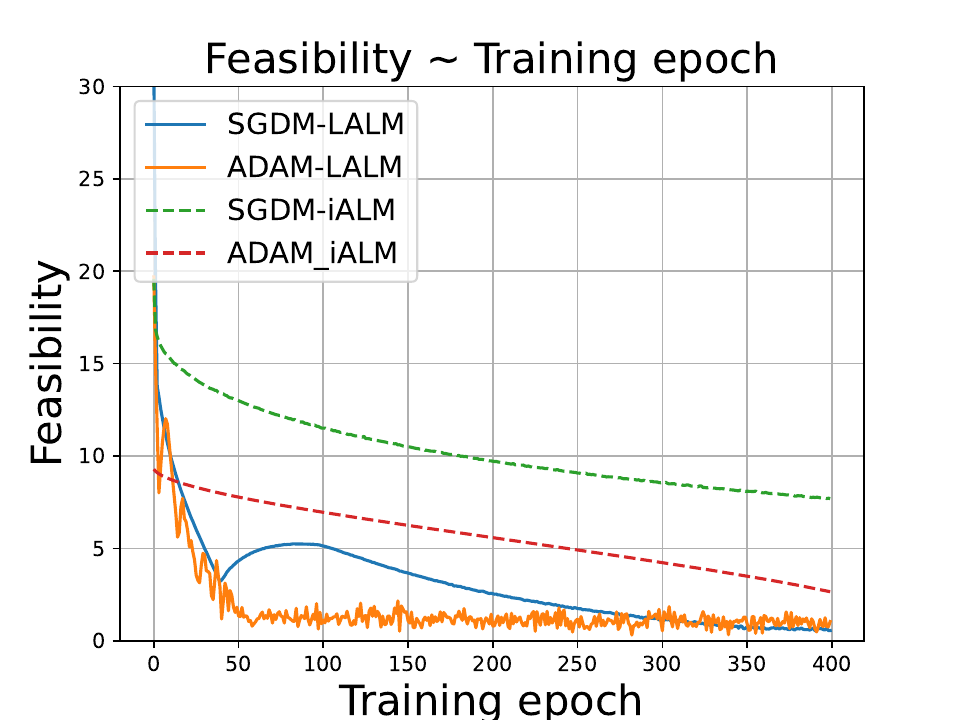}&
					\includegraphics[width=0.33\linewidth]{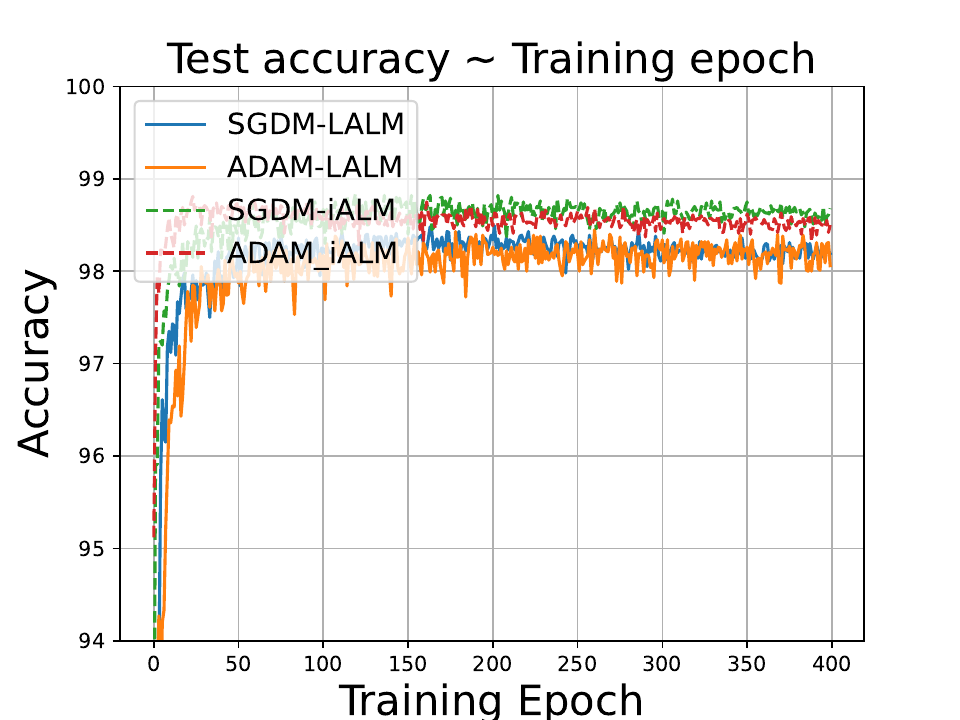}\\
					\multicolumn{1}{c}{\footnotesize{(a) Train loss}} &  \multicolumn{1}{c}{\footnotesize{(b) Feasibility}}&
					\multicolumn{1}{c}{\footnotesize{(c) Test accuracy}}                  
			\end{tabular}}
		\end{center}
	\end{figure*}

	\begin{figure*}[tb]
		\caption{Numerical results on training ResNet14 on CIFAR10 with constraints.} \label{fig:lenet-cifar10}
		\begin{center}
			\setlength{\tabcolsep}{0.0pt}  
			\scalebox{1}{\begin{tabular}{ccc}
					\includegraphics[width=0.33\linewidth]{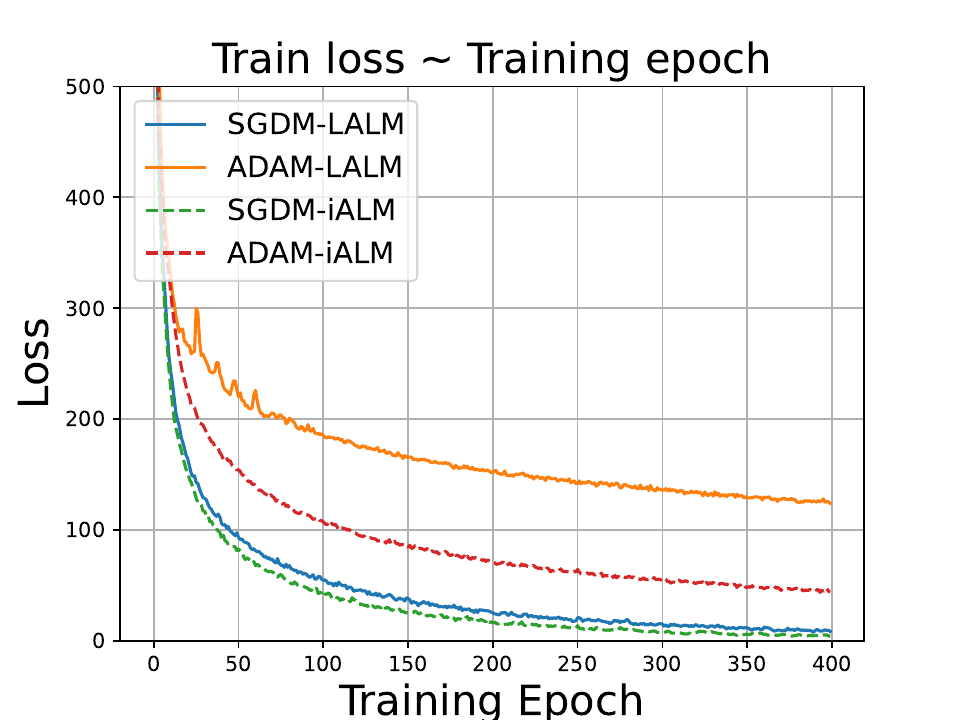}&
					\includegraphics[width=0.33\linewidth]{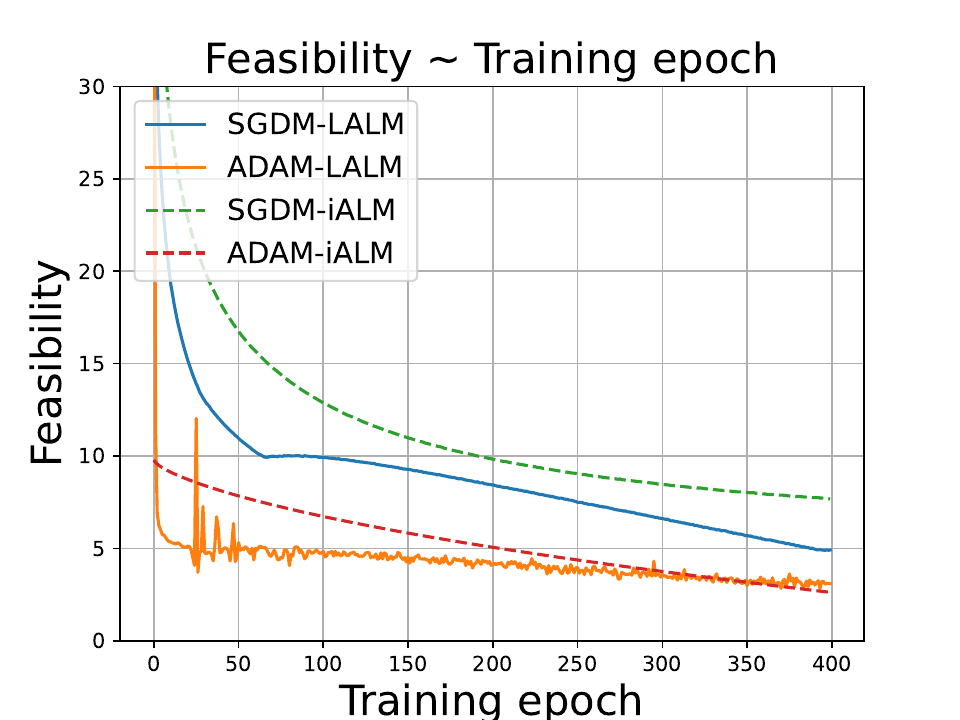}&
					\includegraphics[width=0.33\linewidth]{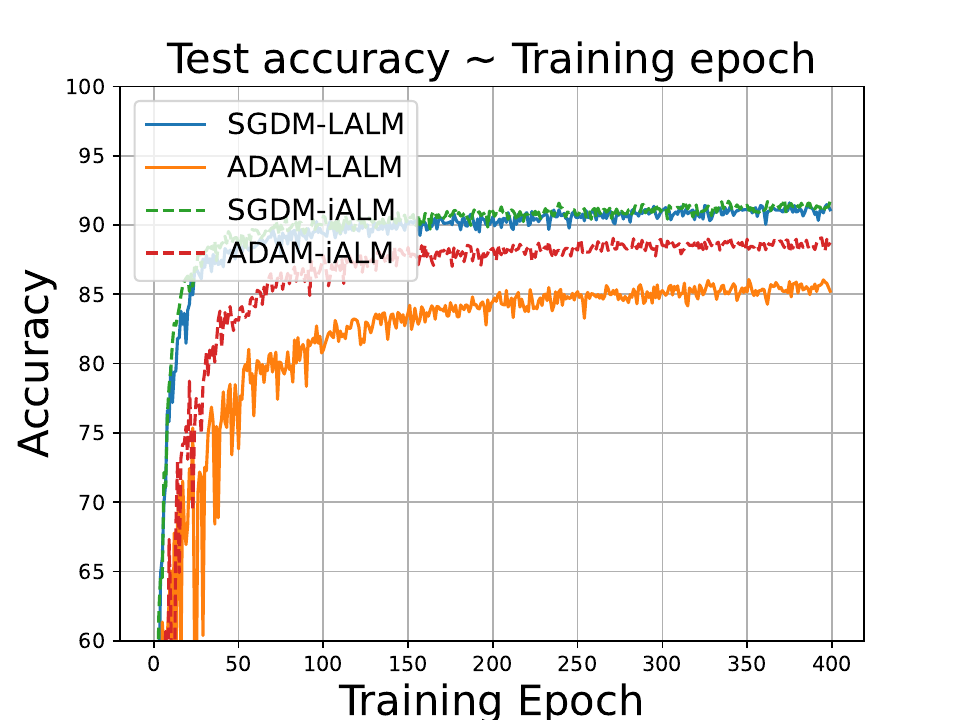}\\
					\multicolumn{1}{c}{\footnotesize{(a) Train loss}} &  \multicolumn{1}{c}{\footnotesize{(b) Feasibility}}&
					\multicolumn{1}{c}{\footnotesize{(c) Test accuracy}}                  
			\end{tabular}}
		\end{center}
	\end{figure*}

	\begin{figure*}[tbp]
		\caption{SGDM-LALM with different stepsizes $\alpha$ to train LeNet10 on MNIST. Other parameters are fixed.} \label{fig:lenet-mnist-sgdm-trade-off}
		\begin{center}
			\setlength{\tabcolsep}{0.0pt}  
			\scalebox{1}{\begin{tabular}{ccc}
					\includegraphics[width=0.33\linewidth]{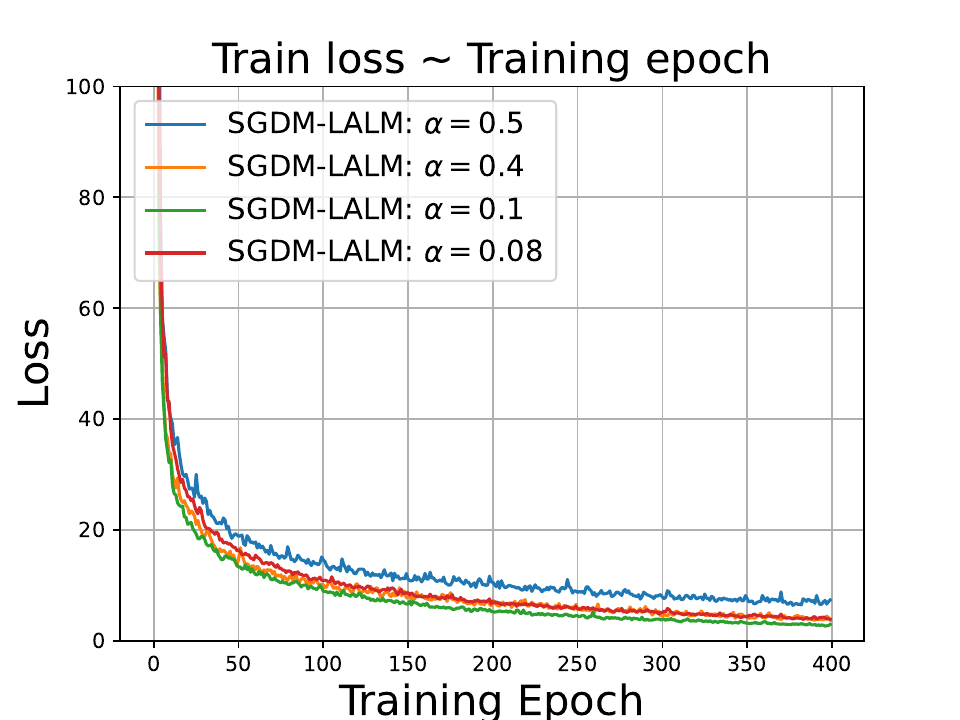}&
					\includegraphics[width=0.33\linewidth]{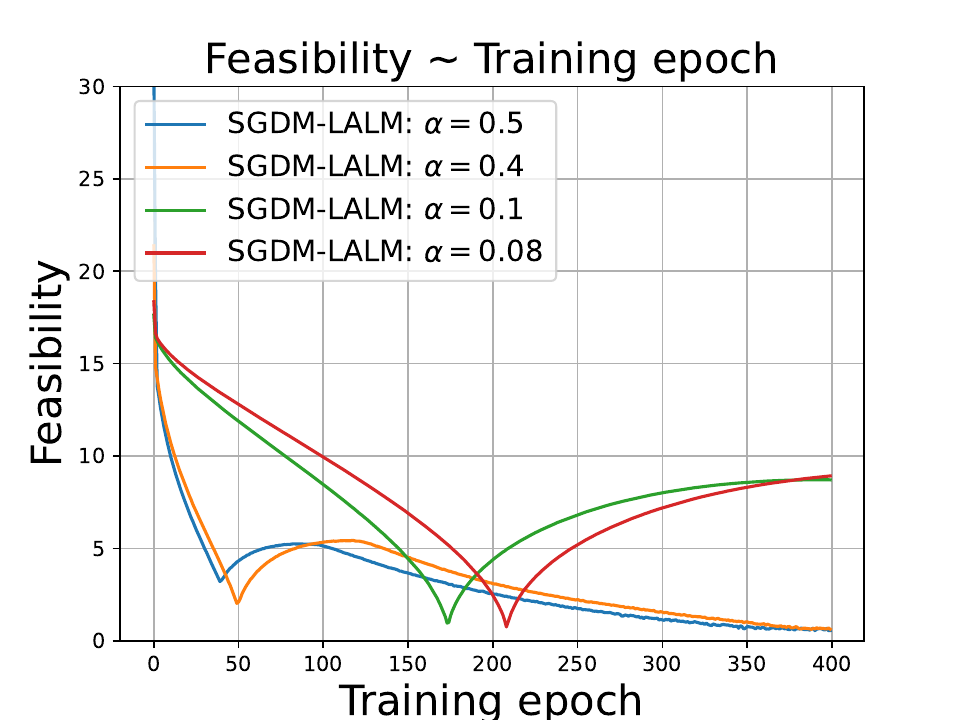}&
					\includegraphics[width=0.33\linewidth]{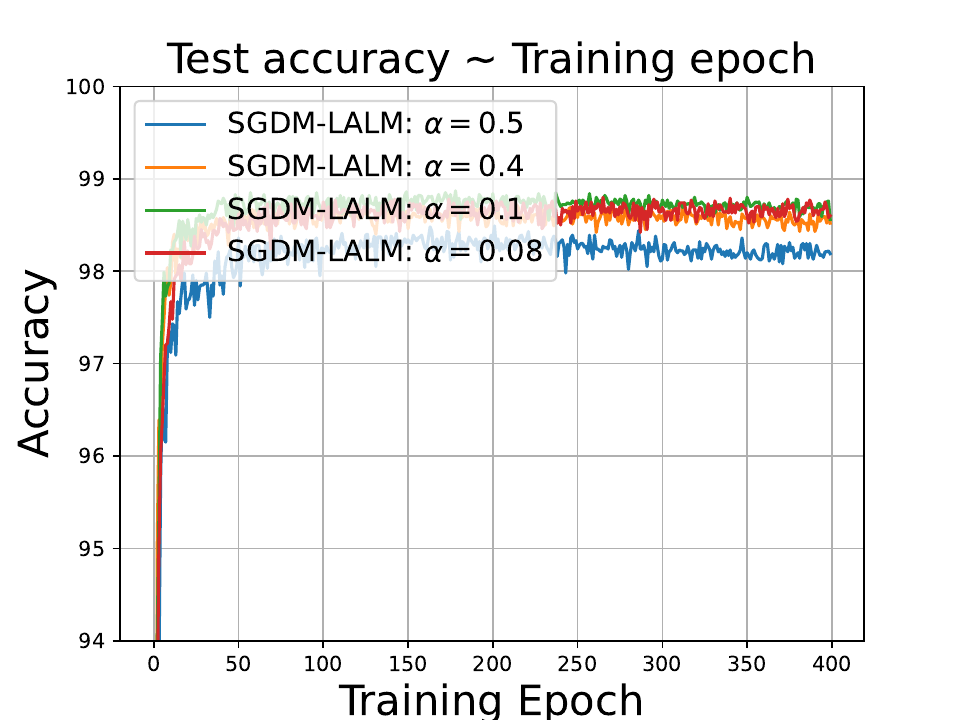}\\
					\multicolumn{1}{c}{\footnotesize{(a) Train loss}} &  \multicolumn{1}{c}{\footnotesize{(b) Feasibility}}&
					\multicolumn{1}{c}{\footnotesize{(c) Test accuracy}}                  
			\end{tabular}}
		\end{center}
	\end{figure*}

	\begin{figure*}[tbh]
		\caption{ADAM-LALM with different $\rho$ to train LeNet10 on MNIST. Other parameters are fixed.} \label{fig:lenet-mnist-adam-trade-off}
		\begin{center}
			\setlength{\tabcolsep}{0.0pt}  
			\scalebox{1}{\begin{tabular}{ccc}
					\includegraphics[width=0.33\linewidth]{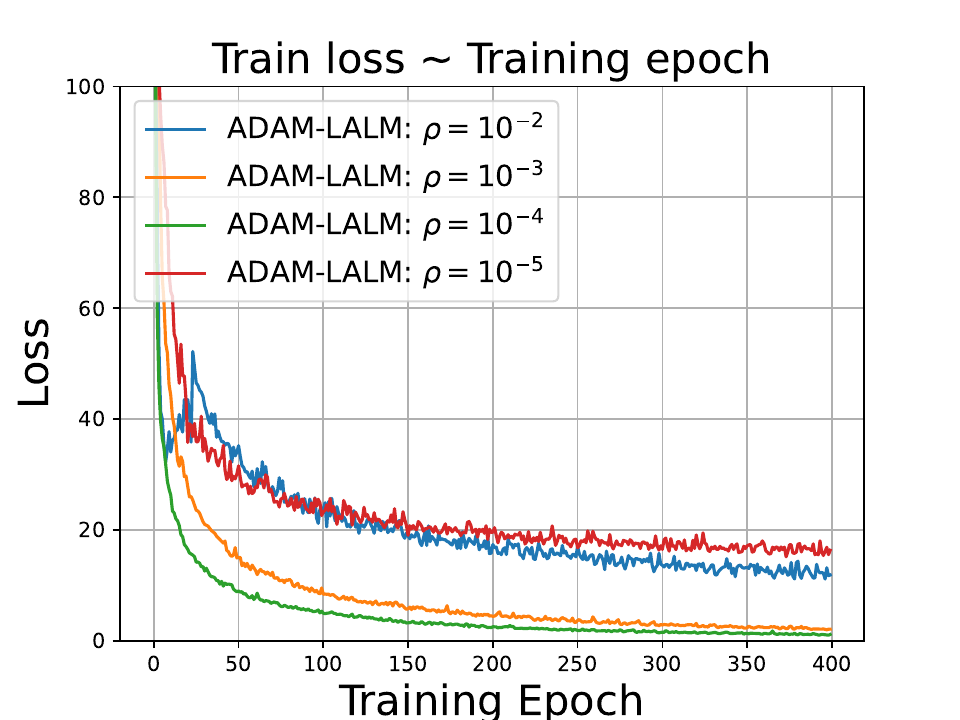}&
					\includegraphics[width=0.33\linewidth]{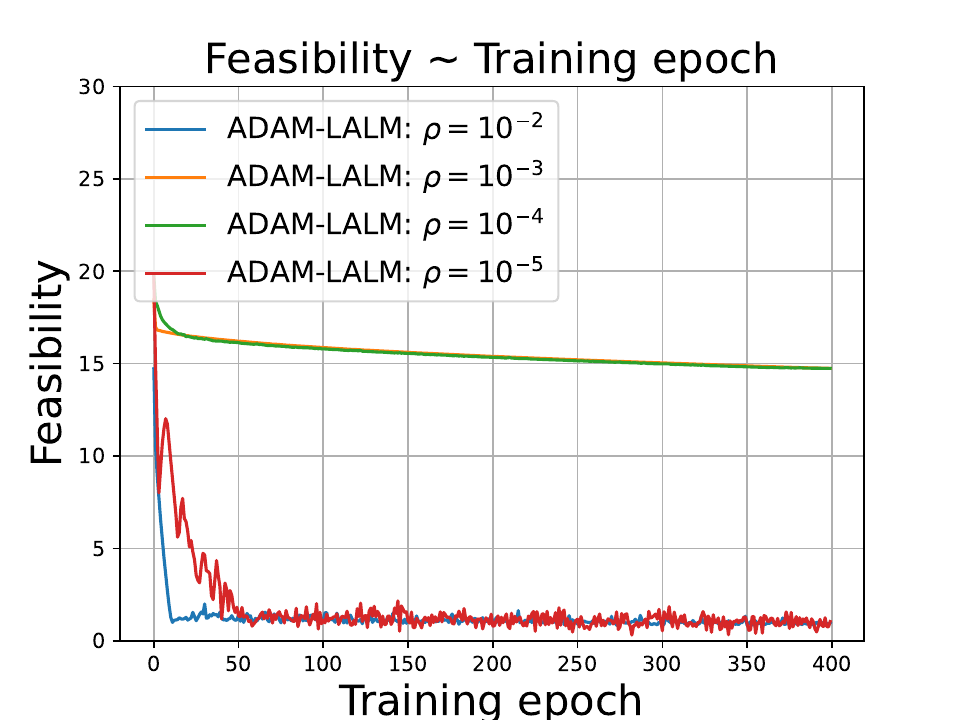}&
					\includegraphics[width=0.33\linewidth]{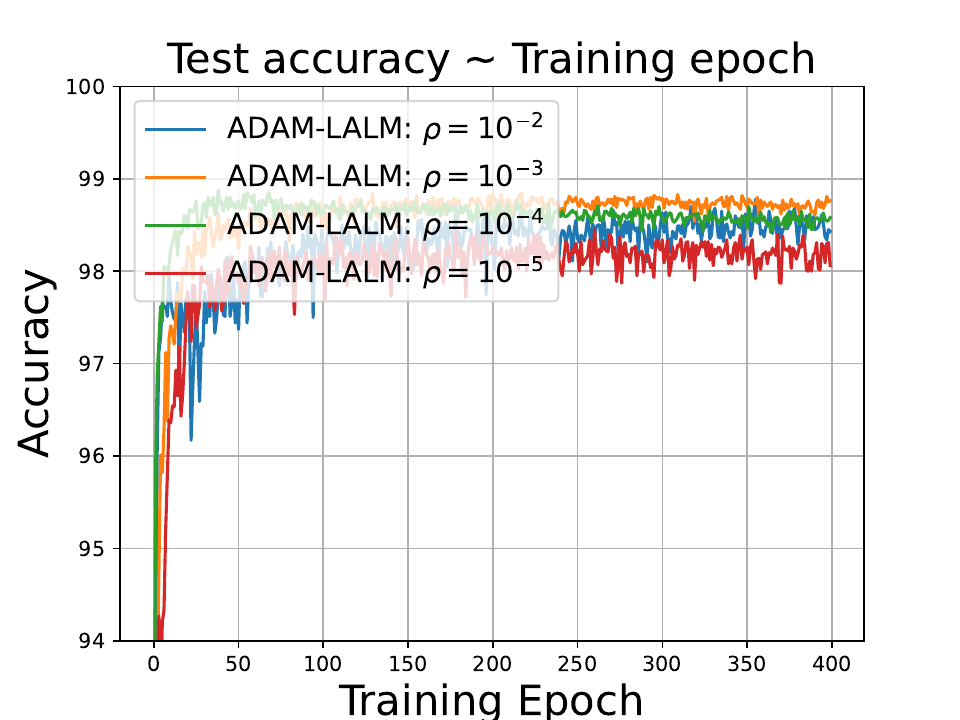}\\
					\multicolumn{1}{c}{\footnotesize{(a) Train loss}} &  \multicolumn{1}{c}{\footnotesize{(b) Feasibility}}&
					\multicolumn{1}{c}{\footnotesize{(c) Test accuracy}}                  
			\end{tabular}}
		\end{center}
	\end{figure*}

	\section{Conclusion}
	
	In this paper, we focus on constrained nonconvex nonsmooth optimization problems in the form of \eqref{Prob_Ori}, where the objective function and constraints are nonconvex and nonsmooth. Motivated by the modified augmented Lagrangian penalty function in \eqref{Eq_Intro_modified_ALM_function},  we introduce a framework \eqref{Eq_Framework} that enables the direct embedding of a wide range of stochastic subgradient methods into Lagrangian-based methods.
	In each iteration of the framework \eqref{Eq_Framework}, we alternatively take a single stochastic subgradient step to update the primal variables, while taking a modified ascent step for the dual variables. We prove the global convergence of the framework \eqref{Eq_Framework} to the KKT points of \eqref{Prob_Ori} in the sense of conservative Jacobians under mild conditions. Moreover, we illustrate that a wide range of efficient subgradient methods, including proximal SGD, proximal SGDM, and proximal ADAM, can be embedded into Lagrangian-based methods through our proposed framework \eqref{Eq_Framework}. Preliminary numerical experiments demonstrate that when embedded with these efficient subgradient methods, Lagrangian-based methods exhibit high efficiency in training neural networks with constraints. These results further demonstrate that by embedding the existing efficient subgradient methods into Lagrangian-based methods, our proposed framework yields efficient optimization approaches for constrained nonconvex nonsmooth optimization problem \eqref{Prob_Ori} with convergence guarantees.

	\appendix
	
	\section{Proofs for the Main Results}
	\label{Appendix_proof_Theo_convergence}
	
	We begin our proof with Lemma \ref{Le_UB_lk}, where we prove that the sequence $\{l_{x,k}\}$ in \eqref{Eq_Framework} is uniformly bounded. 
	\begin{lem}
		\label{Le_UB_lk}
		For any sequence of set-valued mappings $\{\Phi_k\}$ that represents an embeddable stochastic subgradient method, suppose Assumption \ref{Assumption_Phi}, Assumption \ref{Assumption_f} and Assumption \ref{Assumption_framework} hold. Then for the framework \eqref{Eq_Framework}, it holds that 
		\begin{equation}
			\sup_{k\geq 0} \norm{l_{x,k}} < +\infty.
		\end{equation}
	\end{lem}
	\begin{proof}
		Notice that both $\D_f$ and $\D_c$ are locally bounded. Then based on Assumption \ref{Assumption_framework}(1)-(2), we can conclude from the boundedness of $\{\xk\}$ that   
		\begin{equation}
			\label{Eq_Le_UB_lk_0}
			\sup_{k\geq 0}\norm{d_k} < +\infty, \quad \sup_{k\geq 0} \norm{J_{c,k}} < +\infty, \quad \sup_{k\geq 0} \norm{w_k} < +\infty.
		\end{equation}
		Moreover, it follows from the update scheme of $\{\lambda_k\}$ that $\norm{\lambda_{k+1}} \leq \left(1-\frac{\theta_k}{\beta}\right)\norm{\lambda_k} + \theta_k$, hence  
        \begin{equation}
            \norm{\lambda_{k+1}} - \beta \leq \left(1-\frac{\theta_k}{\beta}\right) (\norm{\lambda_{k}} - \beta).
        \end{equation}
        Then Assumption \ref{Assumption_framework}(4) illustrates that $\limsup_{k\geq 0} \norm{\lambda_k} \leq \beta$. 
        Together with \eqref{Eq_Le_UB_lk_0}, we can conclude that   $\sup_{k\geq 0} \norm{l_{x,k}} < +\infty$. This completes the proof. 
	\end{proof}

	Then we have the following lemma showing that $\{\lambda_k\}$ asymptotically approximates $\{\Tregu(\xk)\}$. 
	\begin{lem}
		\label{Le_close_lambda_Tregu}
		For any sequence of set-valued mappings $\{\Phi_k\}$ that represents an embeddable stochastic subgradient method, suppose Assumption \ref{Assumption_Phi}, Assumption \ref{Assumption_f} and Assumption \ref{Assumption_framework} hold. Then for any sequence $\{k_j\}$ such that $k_j \to +\infty$ and $ \{x_{k_j}\}$ converges to  $\bar{x} \in \X$,  it holds that 
		\begin{equation}
			\lim_{j\to \infty} \mathrm{dist}\left( \lambda_{k_j}, \beta \Tregu(c(\bar{x}))   \right) = 0. 
		\end{equation}
	\end{lem}
	\begin{proof}
		The update scheme for $\{\lambda_k\}$ can be rewritten as 
		\begin{equation}
			\lambda_{k+1} \in  \left(1-\frac{\theta_k}{\beta}\right)\lambda_k + \theta_k \Tregu(\wkp). 
		\end{equation}
		Notice that $\{\eta_k\}$ is a diminishing sequence, then for any $k > 0$, there exists a sequence of positive  integers $\{T_k\}$ such that 
		\begin{equation}
			\lim_{k\to \infty}\sum_{i = k - T_k}^k \eta_i = 0, \quad \lim_{k \to \infty} T_k = +\infty. 
		\end{equation}
		Moreover, from \eqref{Eq_condition_Phi}, Lemma \ref{Le_UB_lk}, and the uniform boundedness of $\{\xk\}$ and $\{\yk\}$, there exists $M_c, M_T > 0$ such that  $\norm{c(\xkp) - c(\xk)} \leq \eta_k M_c$ and  $\norm{\Phi_k(l_{x,k}, \xk, \yk, \eta_k) - (\xk, \yk)} \leq \eta_l M_T $ hold for all $k\geq 0$.  
		
		Let $\delta_k^* = \sup_{k-T_k\leq i\leq k}\delta_i + 2\beta (1-\frac{\theta_{\min}}{\beta})^{T_k} + 2\sum_{i = k - T_k}^k \eta_i (1+M_T + M_c)$. Then it is easy to verify that $\lim_{k\to +\infty}\delta_k^* = 0$. Moreover, we have 
		\begin{equation}
			\lambda_k \in \left(\lambda_{k - T_k}\prod_{i = k-T_k }^{k-1} 
			\left( 1 - \frac{\theta_i}{\beta} \right)\right) + \sum_{j = k-T_k + 1}^k \left(\theta_{j-1}\prod_{i = j }^{k-1} \left( 1 - \frac{\theta_i}{\beta} \right)\right) \Tregu(w_j), 
		\end{equation}
		where we define $\Pi_{i=k}^{k-1}\left(1-\frac{\theta_i}{\beta}\right):=1$. As a result, we have
		\begin{equation}
			\mathrm{dist}\left(\lambda_k,  \beta \cdot \mathrm{conv}\left(\bigcup_{\norm{y - c(q_k)} \leq \delta_k^{*}, ~\norm{q_k - \xk}\leq \delta_k^*} \Tregu(y) \right)  \right) \leq 2\beta \left(1-\frac{\theta_{\min}}{\beta}\right)^{T_k} \leq \delta_k^*. 
        \label{eq:A8}
		\end{equation}
		As a result, given any sequence $\{k_j\}$ such that $k_j \to +\infty$ and $\lim_{j\to \infty} x_{k_j} = \bar{x}$ with $\bar{x} \in \X$, for any $j\geq 0$, the Caratheodory's theorem illustrates that there exist $\{y_{k_j, 1},..., y_{k_j, p+1}\}$,  $\{z_{k_j, 1},..., z_{k_j, p+1}\}$ and  $\{\alpha_{k_j, 1},..., \alpha_{k_j, p+1}\}$ such that 
		\begin{equation}
			\begin{aligned}
			    &\mathrm{dist}\left( y_{k_j, i} , \bigcup_{\norm{\tilde{x}-x_{k_j}} \leq \delta^*_{k_j}} c(\tilde{x}) \right)\leq {\delta^*_{k_j}}, \quad &&z_{k_j, i} \in \mathcal{T}_\mathrm{regu}(c(y_{k_j, i})),\\
                &\norm{\lambda_{k_j} - \beta \sum_{i = 1}^{p+1} \alpha_{k_j, i}z_{k_j, i}} \leq \beta \delta^{*}_k, \quad && \sum_{i = 1}^{p+1} \alpha_{k_j,i} = 1,
			\end{aligned}
		\end{equation}
		holds for any $1\leq i\leq p+1$.  
		Therefore, we can conclude that,
		\begin{equation}
			\lim_{j\to \infty}\mathrm{dist} \left((y_{k_j, i}, z_{k_j, i}), (c(x_{k_j}), \Tregu(c(x_{k_j}))) \right) = 0. 
		\end{equation}
		Then from the fact that $\Tregu$ is graph-closed, we can conclude that 
		\begin{equation}
			\lim_{j\to \infty}\mathrm{dist}\left(z_{k_j, i}, \Tregu(c(\bar{x}))  \right) = 0, 
		\end{equation}
		holds for any $1\leq i\leq p+1$. Together with Jensen's inequality, it holds that 
		\begin{equation}
			\begin{aligned}
				&\lim_{j\to \infty}\mathrm{dist}\left(\lambda_{k_j}, \beta \Tregu(c(\bar{x}))  \right) = \lim_{j\to \infty}\mathrm{dist}\left( \sum_{i = 1}^{p+1} \alpha_{k_j, i}z_{k_j, i}, \Tregu(\bar{x})  \right)\\
				\leq{}& \lim_{j\to \infty}\sum_{i = 1}^{p+1} \alpha_{k_j, i}\mathrm{dist}\left( z_{k_j, i}, \Tregu(\bar{x})  \right) = 0. 
			\end{aligned}
		\end{equation}
		This completes the proof. 
	\end{proof}

	Based on Lemma \ref{Le_UB_lk} and Lemma \ref{Le_close_lambda_Tregu}, in the following proposition, we prove that  $\{d_k + J_{c,k}(\lambda_k + \rho \wk ) \}$ can be regarded as a noiseless approximation for  $\D_g$ in \eqref{Eq_Framework}.  
	\begin{prop}
		\label{Prop_lk_Dg}
		For any sequence of set-valued mappings $\{\Phi_k\}$ that represents an embeddable stochastic subgradient method, suppose Assumption \ref{Assumption_Phi}, Assumption \ref{Assumption_f} and Assumption \ref{Assumption_framework} hold. Then for any $\{k_j: j\geq 0\}$ such that $\lim_{j \to \infty} x_{k_j} = \bar{x}$, it holds that 
		\begin{equation}
			\lim_{j\to \infty}\mathrm{dist}\left( d_{k_j}+ J_{c,k_j} (\lambda_{k_j} +  \rho w_{k_j}), \D_g(\bar{x}) \right) = 0. 
		\end{equation}
	\end{prop}
	\begin{proof}
		For any sequence $\{k_j\}$ such that $k_j \to \infty$ and $\lim_{j\to \infty} x_{k_j} = \bar{x}$ with $\bar{x} \in \X$, we first notice that  $d_k \in \D_f^{\delta_k}(\xk)$ and $J_{c, k} \in \D_c^{\delta_k}(\xk)$. Then from the graph-closeness of $\D_f$ and $D_c$, we can conclude that, 
		\begin{equation}
			\lim_{j \to \infty} \mathrm{dist}\left(d_{k_j}, \D_f(\bar{x})  \right) = 0, \quad \lim_{j\to \infty} \mathrm{dist}\left( J_{c, k_j},  \D_{c}(\bar{x})  \right) = 0.
		\end{equation}
		Then Lemma \ref{Le_close_lambda_Tregu} shows that  $\lim_{j\to \infty} \mathrm{dist}\left( \lambda_{k_j}, \beta  \Tregu(c(\bar{x}))   \right) = 0$ holds. 
		Therefore, from the local boundedness of $\D_c$ and $\Tregu$, we can conclude that, 
		\begin{equation}
			\lim_{j\to \infty} \mathrm{dist}\left( J_{c, k_j}\lambda_{k_j}, \beta \cdot \D_{c}(\bar{x})  \Tregu(c(\bar{x}))   \right) = 0. 
		\end{equation}
		In addition, Lemma \ref{Le_UB_lk} illustrates that $\lim_{k\to \infty}\norm{\xkp - \xk} = 0$. Together with the continuity of $c$, it holds that 
		\begin{equation}
			\lim_{j\to \infty} \mathrm{dist}\left( J_{c, k_j}c(x_{k_j+1}), \D_{c}(\bar{x}) c(\bar{x})   \right) = 0. 
		\end{equation}
		Combine all these inequalities together, we conclude from the triangle inequality that
		\begin{equation}
			\begin{aligned}
				&\lim_{j\to \infty}\mathrm{dist}\left( d_{k_j}+ J_{c,k_j} (\lambda_{k_j} +  \rho c(x_{k_j})), \D_g(\bar{x}) \right) \\
				\leq{}& \lim_{j \to \infty} \mathrm{dist}\left(d_{k_j}, \D_f(\bar{x})  \right) + \lim_{j\to +\infty} \mathrm{dist}\left( J_{c, k_j}\lambda_{k_j}, \beta \cdot \D_{c}(\bar{x})  \Tregu(c(\bar{x}))   \right) \\
				& + \rho \cdot \lim_{j\to \infty} \mathrm{dist}\left( J_{c, k_j}c(x_{k_j+1}), \D_{c}(\bar{x}) c(\bar{x})   \right) \\
				={}& 0. 
			\end{aligned}
		\end{equation}
		This completes the proof. 
	\end{proof}

	Notice that $l_{x,k} = d_k + J_{c,k} (\lambda_k + \rho \wk) + \xi_{k+1}$, the sequence $\{l_{x,k}\}$ can be regarded as approximate noisy evaluations of $\D_g$. Therefore, based on Definition \ref{Defin_stable_stochastic_method}, we present the proof for Theorem \ref{Theo_convergence_Dg}.

	\begin{proof}[Proof for Theorem \ref{Theo_convergence_Dg}]
		Following Proposition \ref{Prop_lk_Dg} and the graph-closedness of $\D_g$, we can conclude that there exists a positive sequence $\{\tilde{\delta}_k\}$ such that $\lim_{k\to +\infty} \tilde{\delta}_k = 0$, and 
		\begin{equation}
			d_{k}+ J_{c,k} (\lambda_{k} +  \rho w_{k}) \in \D_g^{\tilde{\delta}_k}(\xk)
		\end{equation}
		holds for all $k\geq 0$.  Moreover, the path-differentiability of $g$ follows from the path-differentiability of $f$ and $c$. In addition, the uniform boundedness of $\{(\xk, \yk)\}$ is guaranteed by Assumption \ref{Assumption_framework}(1). Furthermore, Assumption \ref{Assumption_framework}(3) illustrates that 
            \begin{equation}
			\lim\limits_{s \to \infty} \sup\limits_{s\leq i \leq \Lambda(\lambda_s + T)}\norm{ \sum_{k = s}^{i} \eta_k \xi_{k+1}} =0.
		\end{equation}
            holds for any $T>0$.
            
		Therefore, based on the definition of the embeddable stochastic subgradient methods in Definition \ref{Defin_stable_stochastic_method}, we can conclude that any cluster point of the sequence $\{\xk\}$ lies in $\{x \in \X: 0\in \D_g(x) + \NX(x)\}$, and the sequence of function values $\{g(\xk)\}$ converges. This completes the proof.

	\end{proof}

	\bibliographystyle{plain}
	\bibliography{ref}

\begin{thebibliography}{10}

\bibitem{alacaoglu2023complexity}
Ahmet Alacaoglu and Stephen~J Wright.
\newblock Complexity of single loop algorithms for nonlinear programming with
  stochastic objective and constraints.
\newblock {\em arXiv preprint arXiv:2311.00678}, 2023.

\bibitem{andreani2008augmented}
Roberto Andreani, Ernesto~G Birgin, Jos{\'e}~Mario Mart{\'\i}nez, and
  Mar{\'\i}a~Laura Schuverdt.
\newblock On augmented {L}agrangian methods with general lower-level
  constraints.
\newblock {\em SIAM Journal on Optimization}, 18(4):1286--1309, 2008.

\bibitem{aubin2012differential}
J-P Aubin and Arrigo Cellina.
\newblock {\em Differential inclusions: set-valued maps and viability theory},
  volume 264.
\newblock Springer Science \& Business Media, 2012.

\bibitem{barlaud2021learning}
Michel Barlaud and Fr{\'e}d{\'e}ric Guyard.
\newblock Learning sparse deep neural networks using efficient structured
  projections on convex constraints for green {AI}.
\newblock In {\em 2020 25th international conference on pattern recognition
  (ICPR)}, pages 1566--1573. IEEE, 2021.

\bibitem{beck2014introduction}
Amir Beck.
\newblock {\em Introduction to nonlinear optimization: Theory, algorithms, and
  applications with MATLAB}.
\newblock SIAM, 2014.

\bibitem{benaim2006dynamics}
Michel Bena{\"\i}m.
\newblock Dynamics of stochastic approximation algorithms.
\newblock In {\em Seminaire de probabilites XXXIII}, pages 1--68. Springer,
  2006.

\bibitem{benaim2005stochastic}
Michel Bena{\"\i}m, Josef Hofbauer, and Sylvain Sorin.
\newblock Stochastic approximations and differential inclusions.
\newblock {\em SIAM Journal on Control and Optimization}, 44(1):328--348, 2005.

\bibitem{bianchi2022convergence}
Pascal Bianchi, Walid Hachem, and Sholom Schechtman.
\newblock Convergence of constant step stochastic gradient descent for
  non-smooth non-convex functions.
\newblock {\em Set-Valued and Variational Analysis}, pages 1--31, 2022.

\bibitem{bianchi2021closed}
Pascal Bianchi and Rodolfo Rios-Zertuche.
\newblock A closed-measure approach to stochastic approximation.
\newblock {\em arXiv preprint arXiv:2112.05482}, 2021.

\bibitem{birgin2018augmented}
Ernesto~G Birgin, Gabriel Haeser, and Alberto Ramos.
\newblock Augmented {L}agrangians with constrained subproblems and convergence
  to second-order stationary points.
\newblock {\em Computational Optimization and Applications}, 69:51--75, 2018.

\bibitem{birgin2014practical}
Ernesto~G Birgin and Jos{\'e}~Mario Mart{\'\i}nez.
\newblock {\em Practical augmented Lagrangian methods for constrained
  optimization}.
\newblock SIAM, 2014.

\bibitem{bolte2007clarke}
J{\'e}r{\^o}me Bolte, Aris Daniilidis, Adrian Lewis, and Masahiro Shiota.
\newblock Clarke subgradients of stratifiable functions.
\newblock {\em SIAM Journal on Optimization}, 18(2):556--572, 2007.

\bibitem{bolte2023subgradient}
J{\'e}r{\^o}me Bolte, Tam Le, and Edouard Pauwels.
\newblock Subgradient sampling for nonsmooth nonconvex minimization.
\newblock {\em SIAM Journal on Optimization}, 33(4):2542--2569, 2023.

\bibitem{bolte2021nonsmooth}
J{\'e}r{\^o}me Bolte, Tam Le, Edouard Pauwels, and Tony Silveti-Falls.
\newblock Nonsmooth implicit differentiation for machine-learning and
  optimization.
\newblock {\em Advances in Neural Information Processing Systems}, 34, 2021.

\bibitem{bolte2020mathematical}
J{\'e}r{\^o}me Bolte and Edouard Pauwels.
\newblock A mathematical model for automatic differentiation in machine
  learning.
\newblock {\em Advances in Neural Information Processing Systems},
  33:10809--10819, 2020.

\bibitem{bolte2021conservative}
J{\'e}r{\^o}me Bolte and Edouard Pauwels.
\newblock Conservative set valued fields, automatic differentiation, stochastic
  gradient methods and deep learning.
\newblock {\em Mathematical Programming}, 188(1):19--51, 2021.

\bibitem{bolte2022long}
J{\'e}r{\^o}me Bolte, Edouard Pauwels, and Rodolfo Rios-Zertuche.
\newblock Long term dynamics of the subgradient method for {L}ipschitz path
  differentiable functions.
\newblock {\em Journal of the European Mathematical Society}, 2022.

\bibitem{bolte2022differentiating}
J{\'e}r{\^o}me Bolte, Edouard Pauwels, and Antonio~Jos{\'e} Silveti-Falls.
\newblock Differentiating nonsmooth solutions to parametric monotone inclusion
  problems.
\newblock {\em arXiv preprint arXiv:2212.07844}, 2022.

\bibitem{bolte2018nonconvex}
J{\'e}r{\^o}me Bolte, Shoham Sabach, and Marc Teboulle.
\newblock Nonconvex {L}agrangian-based optimization: monitoring schemes and
  global convergence.
\newblock {\em Mathematics of Operations Research}, 43(4):1210--1232, 2018.

\bibitem{boob2023stochastic}
Digvijay Boob, Qi~Deng, and Guanghui Lan.
\newblock Stochastic first-order methods for convex and nonconvex functional
  constrained optimization.
\newblock {\em Mathematical Programming}, 197(1):215--279, 2023.

\bibitem{borkar2009stochastic}
Vivek~S Borkar.
\newblock {\em Stochastic approximation: a dynamical systems viewpoint},
  volume~48.
\newblock Springer, 2009.

\bibitem{bourkhissi2023complexity}
Lahcen~El Bourkhissi and Ion Necoara.
\newblock Complexity of linearized augmented {L}agrangian for optimization with
  nonlinear equality constraints.
\newblock {\em arXiv preprint arXiv:2301.08345}, 2023.

\bibitem{castera2021inertial}
Camille Castera, J{\'e}r{\^o}me Bolte, C{\'e}dric F{\'e}votte, and Edouard
  Pauwels.
\newblock An inertial {N}ewton algorithm for deep learning.
\newblock {\em The Journal of Machine Learning Research}, 22(1):5977--6007,
  2021.

\bibitem{chen2021solving}
Tianyi Chen, Yuejiao Sun, and Wotao Yin.
\newblock Solving stochastic compositional optimization is nearly as easy as
  solving stochastic optimization.
\newblock {\em IEEE Transactions on Signal Processing}, 69:4937--4948, 2021.

\bibitem{clarke1990optimization}
Frank~H Clarke.
\newblock {\em Optimization and nonsmooth analysis}, volume~5.
\newblock SIAM, 1990.

\bibitem{cohen2022dynamic}
Eyal Cohen, Nadav Hallak, and Marc Teboulle.
\newblock A dynamic alternating direction of multipliers for nonconvex
  minimization with nonlinear functional equality constraints.
\newblock {\em Journal of Optimization Theory and Applications}, pages 1--30,
  2022.

\bibitem{cohen2024alternating}
Eyal Cohen and Marc Teboulle.
\newblock Alternating and parallel proximal gradient methods for nonsmooth,
  nonconvex minimax: A unified convergence analysis.
\newblock {\em Mathematics of Operations Research}, 2024.

\bibitem{curtis2015adaptive}
Frank~E Curtis, Hao Jiang, and Daniel~P Robinson.
\newblock An adaptive augmented {L}agrangian method for large-scale constrained
  optimization.
\newblock {\em Mathematical Programming}, 152(1-2):201--245, 2015.

\bibitem{daniilidis2020pathological}
Aris Daniilidis and Dmitriy Drusvyatskiy.
\newblock Pathological subgradient dynamics.
\newblock {\em SIAM Journal on Optimization}, 30(2):1327--1338, 2020.

\bibitem{davis2020stochastic}
Damek Davis, Dmitriy Drusvyatskiy, Sham Kakade, and Jason~D Lee.
\newblock Stochastic subgradient method converges on tame functions.
\newblock {\em Foundations of Computational Mathematics}, 20(1):119--154, 2020.

\bibitem{duchi2018stochastic}
John~C Duchi and Feng Ruan.
\newblock Stochastic methods for composite and weakly convex optimization
  problems.
\newblock {\em SIAM Journal on Optimization}, 28(4):3229--3259, 2018.

\bibitem{ghadimi2020single}
Saeed Ghadimi, Andrzej Ruszczynski, and Mengdi Wang.
\newblock A single timescale stochastic approximation method for nested
  stochastic optimization.
\newblock {\em SIAM Journal on Optimization}, 30(1):960--979, 2020.

\bibitem{greenstein2023augmented}
Dan Greenstein and Nadav Hallak.
\newblock An augmented {L}agrangian approach for problems with random matrix
  composite structure.
\newblock {\em arXiv preprint arXiv:2305.01055}, 2023.

\bibitem{grigas2019stochastic}
Paul Grigas, Alfonso Lobos, and Nathan Vermeersch.
\newblock Stochastic in-face {F}rank-{W}olfe methods for non-convex
  optimization and sparse neural network training.
\newblock {\em arXiv preprint arXiv:1906.03580}, 2019.

\bibitem{gurbuzbalaban2022stochastic}
Mert G{\"u}rb{\"u}zbalaban, Andrzej Ruszczy{\'n}ski, and Landi Zhu.
\newblock A stochastic subgradient method for distributionally robust
  non-convex and non-smooth learning.
\newblock {\em Journal of Optimization Theory and Applications},
  194(3):1014--1041, 2022.

\bibitem{hallak2023adaptive}
Nadav Hallak and Marc Teboulle.
\newblock An adaptive {L}agrangian-based scheme for nonconvex composite
  optimization.
\newblock {\em Mathematics of Operations Research}, 2023.

\bibitem{he2016deep}
Kaiming He, Xiangyu Zhang, Shaoqing Ren, and Jian Sun.
\newblock Deep residual learning for image recognition.
\newblock In {\em Proceedings of the IEEE conference on computer vision and
  pattern recognition}, pages 770--778, 2016.

\bibitem{hestenes1969multiplier}
Magnus~R Hestenes.
\newblock Multiplier and gradient methods.
\newblock {\em Journal of Optimization Theory and Applications}, 4(5):303--320,
  1969.

\bibitem{hu2022improved}
Xiaoyin Hu, Nachuan Xiao, Xin Liu, and Kim-Chuan Toh.
\newblock An improved unconstrained approach for bilevel optimization.
\newblock {\em SIAM Journal on Optimization}, 33(4):2801--2829, 2023.

\bibitem{jorge2006numerical}
Nocedal Jorge and J~Wright Stephen.
\newblock {\em Numerical optimization}.
\newblock Spinger, 2006.

\bibitem{kingma2014adam}
Diederik~P Kingma and Jimmy Ba.
\newblock Adam: A method for stochastic optimization.
\newblock {\em In Proceedings of the 3rd International Conference for Learning
  Representations}, 2015.

\bibitem{krizhevsky2009learning}
Alex Krizhevsky, Geoffrey Hinton, et~al.
\newblock Learning multiple layers of features from tiny images.
\newblock 2009.

\bibitem{le2023nonsmooth}
Tam Le.
\newblock Nonsmooth nonconvex stochastic heavy ball.
\newblock {\em arXiv preprint arXiv:2304.13328}, 2023.

\bibitem{lecun1998mnist}
Yann LeCun.
\newblock The mnist database of handwritten digits.
\newblock {\em http://yann. lecun. com/exdb/mnist/}, 1998.

\bibitem{lecun1998gradient}
Yann LeCun, L{\'e}on Bottou, Yoshua Bengio, and Patrick Haffner.
\newblock Gradient-based learning applied to document recognition.
\newblock {\em Proceedings of the IEEE}, 86(11):2278--2324, 1998.

\bibitem{li2021rate}
Zichong Li, Pin-Yu Chen, Sijia Liu, Songtao Lu, and Yangyang Xu.
\newblock Rate-improved inexact augmented {L}agrangian method for constrained
  nonconvex optimization.
\newblock In {\em International Conference on Artificial Intelligence and
  Statistics}, pages 2170--2178. PMLR, 2021.

\bibitem{li2023stochastic}
Zichong Li, Pin-Yu Chen, Sijia Liu, Songtao Lu, and Yangyang Xu.
\newblock Stochastic inexact augmented {L}agrangian method for nonconvex
  expectation constrained optimization.
\newblock {\em Computational Optimization and Applications}, pages 1--31, 2023.

\bibitem{lin2022complexity}
Qihang Lin, Runchao Ma, and Yangyang Xu.
\newblock Complexity of an inexact proximal-point penalty method for
  constrained smooth non-convex optimization.
\newblock {\em Computational optimization and applications}, 82(1):175--224,
  2022.

\bibitem{ma2020quadratically}
Runchao Ma, Qihang Lin, and Tianbao Yang.
\newblock Quadratically regularized subgradient methods for weakly convex
  optimization with weakly convex constraints.
\newblock In {\em International Conference on Machine Learning}, pages
  6554--6564. PMLR, 2020.

\bibitem{papadimitriou2023stochastic}
Dimitri Papadimitriou and Bang Vu.
\newblock Stochastic {L}agrangian-based method for nonconvex optimization with
  nonlinear constraints.
\newblock 2023.

\bibitem{pokutta2020deep}
Sebastian Pokutta, Christoph Spiegel, and Max Zimmer.
\newblock Deep neural network training with {F}rank-{W}olfe.
\newblock {\em arXiv preprint arXiv:2010.07243}, 2020.

\bibitem{polyak1964some}
Boris~T Polyak.
\newblock Some methods of speeding up the convergence of iteration methods.
\newblock {\em Ussr computational mathematics and mathematical physics},
  4(5):1--17, 1964.

\bibitem{powell1969method}
Michael~JD Powell.
\newblock A method for nonlinear constraints in minimization problems.
\newblock {\em Optimization}, pages 283--298, 1969.

\bibitem{ruszczynski2020convergence}
Andrzej Ruszczy{\'n}ski.
\newblock Convergence of a stochastic subgradient method with averaging for
  nonsmooth nonconvex constrained optimization.
\newblock {\em Optimization Letters}, 14(7):1615--1625, 2020.

\bibitem{sabach2019lagrangian}
Shoham Sabach and Marc Teboulle.
\newblock Lagrangian methods for composite optimization.
\newblock {\em Handbook of Numerical Analysis}, 20:401--436, 2019.

\bibitem{sabach2022faster}
Shoham Sabach and Marc Teboulle.
\newblock {Faster Lagrangian-based methods in convex optimization}.
\newblock {\em SIAM Journal on Optimization}, 32(1):204--227, 2022.

\bibitem{sahin2019inexact}
Mehmet~Fatih Sahin, Ahmet Alacaoglu, Fabian Latorre, Volkan Cevher, et~al.
\newblock An inexact augmented {L}agrangian framework for nonconvex
  optimization with nonlinear constraints.
\newblock {\em Advances in Neural Information Processing Systems}, 32, 2019.

\bibitem{tang2023self}
Tianyun Tang and Kim-Chuan Toh.
\newblock Self-adaptive {ADMM} for semi-strongly convex problems.
\newblock {\em Mathematical Programming Computation}, pages 1--38, 2023.

\bibitem{tang2024solving}
Tianyun Tang and Kim-Chuan Toh.
\newblock Solving graph equipartition {SDPs} on an algebraic variety.
\newblock {\em Mathematical Programming}, 204(1):299--347, 2024.

\bibitem{van1996geometric}
Lou Van~den Dries and Chris Miller.
\newblock Geometric categories and o-minimal structures.
\newblock {\em Duke Mathematical Journal}, 84(2):497--540, 1996.

\bibitem{wang2023strong}
Shiwei Wang, Chao Ding, Yangjing Zhang, and Xinyuan Zhao.
\newblock Strong variational sufficiency for nonlinear semidefinite programming
  and its implications.
\newblock {\em SIAM Journal on Optimization}, 33(4):2988--3011, 2023.

\bibitem{xiao2023adam}
Nachuan Xiao, Xiaoyin Hu, Xin Liu, and Kim-Chuan Toh.
\newblock Adam-family methods for nonsmooth optimization with convergence
  guarantees.
\newblock {\em Journal of Machine Learning Research}, 25(48):1--53, 2024.

\bibitem{xiao2023convergence}
Nachuan Xiao, Xiaoyin Hu, and Kim-Chuan Toh.
\newblock Convergence guarantees for stochastic subgradient methods in
  nonsmooth nonconvex optimization.
\newblock {\em arXiv preprint arXiv:2307.10053}, 2023.

\bibitem{xie2021complexity}
Yue Xie and Stephen~J Wright.
\newblock Complexity of proximal augmented {L}agrangian for nonconvex
  optimization with nonlinear equality constraints.
\newblock {\em Journal of Scientific Computing}, 86:1--30, 2021.

\bibitem{xu2023unified}
Zi~Xu, Huiling Zhang, Yang Xu, and Guanghui Lan.
\newblock A unified single-loop alternating gradient projection algorithm for
  nonconvex--concave and convex--nonconcave minimax problems.
\newblock {\em Mathematical Programming}, pages 1--72, 2023.

\bibitem{yang2024data}
Shuoguang Yang, Xudong Li, and Guanghui Lan.
\newblock Data-driven minimax optimization with expectation constraints.
\newblock {\em Operations Research}, 2024.

\bibitem{zhu2023first}
Daoli Zhu, Lei Zhao, and Shuzhong Zhang.
\newblock A first-order primal-dual method for nonconvex constrained
  optimization based on the augmented {L}agrangian.
\newblock {\em Mathematics of Operations Research}, 2023.

\end{thebibliography}

\end{document}